\documentclass[a4paper,12pt]{amsart}
\usepackage{amssymb}
\usepackage{dsfont}
\usepackage[all,cmtip]{xy}
\usepackage{amsmath}
\usepackage{chemarrow}
\usepackage[colorlinks, 
pdfborder=001, 
linkcolor=cyan,
anchorcolor=magenta
citecolor=green
]{hyperref}
\usepackage{mathrsfs}
\usepackage{wasysym}
\usepackage{titletoc}
\usepackage{bm}
\usepackage{bbm}
\usepackage{bbold} 
\usepackage{cite}
\usepackage{subfiles}
\usepackage{extarrows} 
\usepackage{colordvi}
\usepackage{color}
\usepackage{stmaryrd}
\usepackage{silence}
\allowdisplaybreaks
\DeclareMathOperator{\To}{\longrightarrow}

\DeclareMathOperator{\D}{\mathcal{D}}

\DeclareMathOperator{\Ext}{Ext}

\DeclareMathOperator{\Hom}{Hom}

\DeclareMathOperator{\id}{id}

\DeclareMathOperator{\M}{\mathcal{M}}

\DeclareMathOperator{\RHom}{\textbf{R}Hom}
\DeclareMathOperator{\rhu}{\rightharpoonup}

\title[Gorenstein homological invariants and monoidal model categories]{Gorenstein homological invariants and monoidal model categories of Hopf algebras}
\author{Wei Ren and Ruipeng Zhu}

\numberwithin{equation}{section}

\theoremstyle{definition}

\newtheorem{thm}{Theorem}[section]
\newtheorem{prop}[thm]{Proposition}
\newtheorem{lem}[thm]{Lemma}

\newtheorem{cor}[thm]{Corollary}

\newtheorem{rk}[thm]{Remark}

\newtheorem{ex}[thm]{Example}

\begin{document}

\begin{abstract}
Let $H$ be a Hopf algebra over a field $k$ with a bijective antipode. It is proved that the Gorenstein global dimension of $H$ coincides with the Gorenstein projective dimension of the trivial left (or right) $H$-module $k$. Then, $H$ is finite dimensional if and only if the Gorenstein projective dimension of $k$ is trivial. Although monoidal Morita-Takeuchi equivalence of Hopf algebras does not preserve the global dimension, we demonstrate that it does preserve the Gorenstein global dimension and the Artin-Schelter Gorenstein property; this supports Brown-Goodearl's question of whether every noetherian (affine) Hopf algebra is AS Gorenstein.
Finally, for $H$ and an $H$-Galois object $B$,
we show the categories of modules $_H\mathcal{M}$ and $_B\mathcal{M}_B^H$ are monoidal model categories with respect to the Gorenstein projective model structure, provided that the Gorenstein global dimension of $H$ is finite. The corresponding stable categories are tensor triangulated categories.
\end{abstract}
	
\makeatletter
\@namedef{subjclassname@2020}{\textup{2020} Mathematics Subject Classification}
\makeatother

\date{}
\thanks{Corresponding author (R. Zhu). E-mail address: zhuruipeng$\symbol{64}$sufe.edu.cn.}
\maketitle
\tableofcontents
	
\addtocounter{section}{-1}
\section{Introduction}
\noindent
A module is said to be Gorenstein projective if it is a syzygy of a totally acyclic complex of projective modules \cite{EJ}. In the literature, finitely generated Gorenstein projective modules are also known as modules of G-dimension zero, maximal Cohen-Macaulay modules or totally reflexive modules; see e.g. \cite{AB, Buc, EJ}. The Gorenstein projective dimension of a module $M$, denoted by ${\rm Gpd}(M)$, is defined as the minimal length of a resolution of $M$ by Gorenstein projective modules. This notion has its origins in the work of Auslander and Bridger \cite{AB} on $G$-dimension, which generalizes projective dimension in the sense that ${\rm Gpd}(M)\leq {\rm pd}(M)$, with equality whenever ${\rm pd}(M)$ is finite. For a ring $R$, the Gorenstein global dimension ${\rm Ggldim}(R)$ is defined as the supremum of the Gorenstein projective dimension of all (left) $R$-modules. The terminology is justified in \cite{BM}, where it is shown that ${\rm Ggldim}(R)$ also equals the supremum of the Gorenstein injective dimension of all $R$-modules. Over the past few decades, the Gorenstein homological algebra has experienced a rapid development and found interesting applications across diverse areas, including the representation theory of Artin algebras, singularity categories, cohomology theory of commutative rings and cohomological group theory.

Let $H$ be a Hopf algebra over a field $k$. A profound phenomenon in the homological study of Hopf algebras is that certain homological invariants of $H$ are completely determined by the trivial $H$-module $k$ (defined via the counit $\varepsilon$). Our starting point is an observation by Lorenz and Lorenz \cite[2.4]{LL} on the left (resp. right) global dimension of $H$:
\[{\rm l.gldim}(H)= {\rm pd}_H({_\varepsilon}k),\text{ and } \quad {\rm r.gldim}(H)= {\rm pd}_H(k_\varepsilon).\]
Although the global dimension of any algebra is not necessarily symmetric, for Hopf algebra $H$, there are equalities
${\rm l.gldim}(H)= {\rm pd}_{H^e}(H) ={\rm r.gldim}(H)$, where ${\rm pd}_{H^e}(H)$ denotes the projective dimension of $H$ as an $H$-bimodule, i.e. as a left module over the enveloping algebra $H^e= H\otimes H^{op}$; see e.g. \cite{Bi} or \cite[Appendix]{WYZ}. 

Throughout, we assume that the antipode of the Hopf algebras is bijective.  Note that this assumption is not too restrictive; see e.g. \cite{BZ, Sk1}. Motivated by work of Emmanouil and Talelli \cite{ET1, ET3} on Gorenstein cohomological dimension of groups, we show in Theorem \ref{thm:fGpd} that the trivial module is a test module for the Gorenstein global dimension of Hopf algebra $H$. Moreover, we obtain the Gorenstein counterpart of the above equalities in Corollary \ref{cor:equ-2}.

\begin{thm}\label{thm-1}
(i) Let $H$ be a Hopf algebra over a field $k$. The following are equivalent:
	\begin{enumerate}
	\item ${\rm Gpd}_H({_\varepsilon}k)<\infty$.
	\item There is an $H$-monomorphism $\iota: {_\varepsilon}k\rightarrow \Lambda$
	with ${\rm pd}_H(\Lambda)<\infty$.
 \item For any $H$-module $M$, ${\rm Gpd}_H(M)<\infty$.
 \end{enumerate}
\noindent If these equivalent conditions hold, then ${\rm Gpd}_H({_\varepsilon}k)={\rm pd}_H(\Lambda)$, where $\Lambda$ is as in \textup{(2)}.

(ii)  ${\rm l.Ggldim}(H)= {\rm Gpd}_H({_\varepsilon}k) ={\rm Gpd}_{H^e}(H)= {\rm r.Ggldim}(H)= {\rm Gpd}_H(k_\varepsilon)$.
\end{thm}

In \cite{LZ}, Liu and Zhang gave an affirmative answer to a question posed by Bergen, showing that an artinian Hopf algebra over a field is finite dimensional. As an application of the above theorem, we extend this characterization of finite dimensional Hopf algebras (Proposition \ref{prop:f.d.H}) by proving that the following conditions are equivalent:\\
\indent (1) $H$ is finite dimensional; (2) $H$ is left (or right) artinian; (3) ${\rm Gpd}_H(k)= 0$;\\ \indent (4) As a left (or right) $H$-module, $k$ is a submodule of a projective module.\\
\noindent Theorem \ref{thm:fGpd} also leads to several further consequences concerning classical homological invariants and Hopf algebra extensions. We obtain equality between ${\rm l.silp}(H)$ and ${\rm l.spli}(H)$ under mild countability hypotheses, and in the noetherian case identify these invariants with both ${\rm l.Ggldim}(H)$ and ${\rm id}_H(H)$; see Propositions~\ref{silp=spli-prop}, \ref{prop:silp=spli+1} and Corollary \ref{silp=spli-noe-Hopf-cor}. We also establish restriction and extension inequalities for the Gorenstein projective dimension along left coideal subalgebras, together with an additivity formula for normal extensions under suitable finiteness assumptions; see Corollaries \ref{cor:Gpd-subalg} and \ref{cor:Gpd-quoalg}.

\vspace{0.1in}
 
A prominent problem in the study of Hopf algebras is whether important homological invariants are preserved under various forms of ``deformations''. Two Hopf algebras $H$ and $L$ are said to be monoidally Morita-Takeuchi equivalent, if their comodule categories are monoidally equivalent; this condition is known to be equivalent to the existence of an $L$-$H$-biGalois extension of the base field $k$ \cite{S1}. 
As shown in recent literature, several important homological invariants, such as the skew Calabi-Yau property and the global dimension of homologically smooth Hopf algebras, have been examined under monoidal Morita-Takeuchi equivalence; see e.g. \cite{Bi, WYZ, Zh}.
In contrast,  by \cite[Example 4.11]{Zh} monoidally Morita-Takeuchi equivalences do not preserve the global dimension without the smoothness hypothesis.

A natural extension of this line of inquiry is to ask whether the Gorenstein global dimension is preserved under monoidal Morita-Takeuchi equivalence. Let $H$ be a Hopf algebra, and let $B$ be a right $H$-comodule algebra. We consider the Ehresmann-Schauenburg bialgebroid $(B^e)^{co H}$ \cite{BW}. In view of this question, the main result of Section 3 gives an affirmative answer under suitable hypotheses; see Theorem \ref{mMTe-thm}. 
 
 \begin{thm}\label{thm-2}
Let $H$ and $L$ be two monoidally Morita-Takeuchi equivalent Hopf algebras, and let $B$ be an $L$-$H$-biGalois extension of $k$. Then
\[ {\rm Ggldim}(H) = {\rm Gpd}_H(k) = {\rm Gpd}_{B^e}(B) = {\rm Gpd}_L(k) = {\rm Ggldim}(L). \]
\end{thm}

A Hopf algebra $H$ is called {\it Artin-Schelter Gorenstein}, or {\it AS Gorenstein} for short, if $_HH$ has finite injective dimension, say $d$, and $\Ext^i_H({}_{\varepsilon}k, H) = 0$ for all $i \neq d$ and $\Ext^d_H({}_{\varepsilon}k, H) = k$.
It is well-known that any noetherian AS Gorenstein Hopf algebra is a Gorenstein algebra, i.e. an algebra with finite Gorenstein global dimension. However,
there exist Gorenstein Hopf algebras that are not AS Gorenstein; see Example \ref{not-AS-G-eg}. The following question, often attributed to Brown and Goodearl, remains open and continues to attract considerable attention; see e.g. \cite[1.15]{BG}, \cite[Question A]{Br} and \cite[1.2]{BZ}.\vspace{0.1in} \\
\noindent {\bf Question}: Is every noetherian (affine) Hopf algebra AS Gorenstein?\vspace{0.1in}\\
\noindent  As an application, we have the following: (1) gives an affirmative answer, under the noetherian hypothesis, to a question raised in \cite{Zh}, and (2) supports Brown-Goodearl's question; see Proposition~\ref{ID-H-L-cor}.

\begin{prop}\label{prop-3}
Let $H$ and $L$ be two monoidally Morita-Takeuchi equivalent noetherian Hopf algebras. Then the following hold.
    \begin{enumerate}
    \item ${\rm id}_H(H) = {\rm id}_L(L)$.
    \item $H$ is AS Gorenstein if and only if $L$ is AS Gorenstein.
    \end{enumerate}
\end{prop}

The notion of a model category was introduced by Quillen to provide an axiomatization of homotopy theory. 
A monoidal model category is a model category equipped with a compatible monoidal structure, ensuring that its homotopy category inherits a tensor triangulated structure; we refer to \cite{H1} for details.

Let $H$ be a Hopf algebra over a field $k$. It is standard that the category of left $H$-modules $_H\mathcal{M}$ forms a monoidal category under tensor product over $k$, with $k$ as the unit object. If ${\rm Gpd}_H({_\varepsilon}k)={\rm l.Ggldim}(H)$ is finite, then in analogy with \cite[Theorem 8.6]{H2}, the category $_H\mathcal{M}$ admits a Gorenstein projective model structure. The associated homotopy category obtained by formally inverting the weak equivalences, is then triangle-equivalent to the stable category $\underline{\rm GProj}(H)$ of Gorenstein projective $H$-modules. Moreover, ${_B\!\M^H_B}$ is a monoidal category with tensor product over $B$, and there is a Gorenstein projective model structure on ${_B\!\M^H_B}$ as well; see Lemma \ref{lem:mcs-B-B}. In Section 4, we shall prove the following results; the case of $_H\mathcal{M}$ is Theorem \ref{thm:m-m-cat}, and the case of ${_B\!\M^H_B}$ is Theorem \ref{thm:mMTe-mmc}.

\begin{thm}\label{thm-4}
Assume ${\rm Gpd}_H({_\varepsilon}k)$ is finite. Then the categories $_H\mathcal{M}$ and  ${_B\!\M^H_B}$, endowed with the Gorenstein projective model structure, are monoidal model categories. Moreover, their stable categories of Gorenstein projective modules are tensor triangulated categories, with the tensor product taken over $k$ for $H$-modules and  over $B$ for modules in ${_B\!\M^H_B}$.
\end{thm}

\noindent It is worth noting that the unit object $k$ (resp. $B$) of the tensor product need not be Gorenstein projective. Consequently, the existence of a tensor unit in the stable category is not automatic. It turns out that such a tensor unit is provided by a special right Gorenstein projective approximation of $k$  (resp. of $B$); in view of the Gorenstein projective model structure, this approximation coincides with the cofibrant replacement of $k$  (resp. of $B$).
 
\section{Preliminaries}
\noindent
In this section, we record certain prerequisite notions and facts which are
used throughout the paper, including definitions concerning modules over Hopf algebras, Hopf-Galois extensions, Gorenstein projective modules and model categories.

Throughout this paper, $k$ is a field and all algebras are over $k$. Unadorned $\otimes$ means $\otimes_k$. For any algebra $A$, $A^{op}$ is the opposite algebra of $A$, and $A^e:= A\otimes A^{op}$ is the enveloping algebra of $A$. Usually we work with left modules. A right $A$-module is viewed as an $A^{op}$-module, and an $A$-$A$-bimodule is the same as an $A^e$-module.
\vspace{0.1in}

\noindent
{\sc I.\ Modules over Hopf algebras.}
Let \((H, m, u, \Delta,\varepsilon)\)
be a Hopf algebra over $k$ with counit $\varepsilon: H\rightarrow k$ and antipode
$S: H\rightarrow H$. We will use Sweedler notation for the comultiplication
$\Delta: H\rightarrow H\otimes  H$, that is $\Delta(h)=\sum h_{1}\otimes h_{2}$
for any $h\in H$. Unless otherwise specified $H$-modules will be left modules.
We refer the reader to \cite{M} as a basic reference on Hopf algebras.

Let $M$ and $N$ be any $H$-modules. The tensor product $M\otimes  N$ will be considered
to be an $H$-module via the comultiplication $\Delta$, i.e.
$h(m\otimes n) = \sum h_1 m \otimes h_2 n$. Also, ${\rm Hom}_k(M,N)$ is an $H$-module
via $(h \rhu f)(m)=\sum h_{1}f(S(h_{2})m)$ for any $f\in {\rm Hom}_k(M,N)$.
Similarly, if $M$ and $N$ are right $H$-modules, then $M\otimes  N$ and
${\rm Hom}_{k}(M,N)$ are right $H$-modules. It is clear that
${\rm Hom}_H(M,N)$ is precisely the subspace of $H$-invariant elements of ${\rm Hom}_{k}(M,N)$,
i.e.
${\rm Hom}_k(M,N)^H:=\{f\in{\rm Hom}_k(M,N) \mid h \rhu f=\varepsilon(h)f\text{ for all }h\in H\}$.

The following facts are standard; see \cite[Lemma 1.1]{BG} for (1) and \cite[1.2]{BG} for (2).

\begin{lem}\label{lem:fact}
\begin{enumerate}
    \item Let $M$, $N$ and $L$ be $H$-modules. Then the natural adjoint isomorphism
${\rm Hom}_{k}(M\otimes  N, L)\cong {\rm Hom}_{k}(M, {\rm Hom}_{k}(N , L))$
restricts to an isomorphism
	\[ {\rm Hom}_{H}(M\otimes  N, L)\cong {\rm Hom}_{H}(M, {\rm Hom}_{k}(N , L)).
	\]
 \item Let $P$ be a projective $H$-module. For any $H$-module $M$,
	$P\otimes M$ is a projective $H$-module.
\end{enumerate}
\end{lem}

In general, the flip map $\tau: M\otimes N\rightarrow N\otimes M$, defined by $\tau(m\otimes n)=n\otimes m$, is not necessarily an $H$-isomorphism unless $H$ is cocommutative. However, we have the following.
 
\begin{lem}\label{lem:fact'}
Suppose the antipode of $H$ is bijective. The following hold.
\begin{enumerate}
    \item Let $M$, $N$ and $L$ be $H$-modules, with the $H$-module structure on ${\rm Hom}_{k}(M, L)$ defined by $(h \rightharpoondown f)(m) = \sum h_{2}f(S^{-1}(h_{1}) m)$. Then the natural adjoint isomorphism
${\rm Hom}_{k}(M\otimes  N, L)\cong {\rm Hom}_{k}(N, {\rm Hom}_{k}(M, L))$
restricts to an isomorphism
	\[ {\rm Hom}_{H}(M\otimes N, L)\cong {\rm Hom}_{H}(N, {\rm Hom}_{k}(M, L)).
	\]
  \item Let $P$ be a projective $H$-module. For any $H$-module $M$, $M\otimes P$ is a projective $H$-module.
\end{enumerate}
\end{lem}

If $R$ is a right $H$-comodule algebra via $\rho_R: R\rightarrow R\otimes H$, then a relative right-right $(R,H)$-Hopf module is a vector space with a right $R$-action and a right $H$-coaction $\rho_M$ such that
$\rho_M(mr)=\rho_M(m)\rho_R(r)=\sum m_0 r_0\otimes m_1 r_1$ for all $m\in M$ and $r\in R$. A relative left-right $(R,H)$-Hopf module is defined similarly.
The category of relative right-right (resp., left-right) $(R,H)$-Hopf modules is denoted by  $\mathcal{M}_R^H$ (resp., $_R\mathcal{M}^H$).
If $B$ is another right $H$-comodule algebra, the category of relative $(R, B, H)$-Hopf bimodules (equipped with left $R$-action, right $B$-action, and right $H$-coaction satisfying appropriate compatibilities) is denoted by $_R\mathcal{M}_H^H$ following the same naming convention.
\vspace{0.1in}

\noindent
{\sc II. Hopf-Galois extensions and Schauenburg's structure theorem.}
Let $H$ be a Hopf algebra with a bijective antipode $S$. Let $B$ be a right $H$-comodule algebra with the structure map $\rho: B\rightarrow B\otimes H$, $b\mapsto \sum b_0\otimes b_1$. Note that $\rho$ is an algebra homomorphism. The coinvariant subalgebra of the $H$-coaction on $B$ is $B^{co H}=\{b\in B\,|\, \rho(b)=b\otimes 1\}$. We denote by
$A=B^{co H}$, and $A\subseteq B$ is said to be an $H$-extension.
For an $H$-extension $A\subseteq B$,  $B^{op}$ is a left $H$-comodule algebra with the structure map  $\rho': B^{op}\rightarrow H\otimes B^{op}$, $b\mapsto \sum S^{-1}b_1\otimes b_0$,  and furthermore, we have $A^{op}={}^{co H}(B^{op})$.

Associated to an $H$-extension $A\subseteq B$, there is a canonical Galois map
\[\beta: B\otimes_A B\rightarrow B\otimes H, \quad b\otimes_A b'\mapsto\sum bb'_0\otimes b'_1.\]
If $\beta$ is bijective, then $A \subseteq B$ is called a right $H$-Galois extension. A left $H$-Galois extension  is defined in a similar manner. Consider the map $\tau: B\otimes H\rightarrow H\otimes B$ given by $b\otimes h\mapsto \sum hS^{-1}b_1\otimes b_0$. This is bijective with $\tau^{-1}(h\otimes b)= \sum b_0\otimes hb_1$. The composition $\beta'=\tau \beta: B\otimes_A B\rightarrow H\otimes B$ is given by $\beta'(b\otimes_A b')= \sum S^{-1}b_1\otimes b_0 b'$. So the $H$-extension $A\subseteq B$ is Galois if and only if $\beta'$ is bijective. The translation map associated to the
$H$-Galois extension $A \subseteq B$ is given by 
\[ \kappa: H \rightarrow B \otimes_A B, \quad h \mapsto \beta^{-1}(1 \otimes h). \]
For any $h \in H$, we adopt the Sweedler-like notation $\kappa(h) = \sum \kappa^1(h) \otimes_A \kappa^2(h)$ to denote the image of $h$ under $\kappa$.

In particular, for an $H$-Galois extension $k \subseteq B$, the algebra $B$ is referred to as an $H$-Galois object. Given another Hopf algebra $L$, an $L$-$H$-bi-Galois object is defined as an $L$-$H$-bicomodule algebra that is simultaneously a left $L$-Galois object and a right $H$-Galois object.

The following is immediate from \cite[Theorem I]{S} and \cite[Theorem 5.6]{SS}.

\begin{thm}\label{ffHGe-thm}
Let $H$ be a Hopf algebra with a bijective antipode, $B$ be a right $H$-comodule algebra and $A := B^{co H}$. Then the following are equivalent.
\begin{enumerate}
	\item $(- \otimes_A B, (-)^{co H})$ is an equivalence between $\M_A$ and $\M^H_B$.
     \item $(B \otimes_A -, (-)^{co H})$ is an equivalence between $_A\!\M$ and $_B\!\M^H$.
     \item $B$ is a faithfully flat left $A$-module, and $A \subseteq B$ is an $H$-Galois extension.
     \item $B$ is a faithfully flat right $A$-module, and $A \subseteq B$ is an $H$-Galois extension.
     \item $A \subseteq B$ is a Hopf Galois extension, and the comodule algebra $B$ is right $H$-equivariantly projective as a left $A$-module, i.e., there exists a right $H$-colinear and left $A$-linear splitting of the multiplication $A \otimes B \to B$.
\end{enumerate}
\end{thm}

Two Hopf algebras $H$ and $L$ are said to be monoidally Morita--Takeuchi equivalent, if their
comodule categories are monoidally equivalent; this condition is known to be equivalent to the existence of an $L$-$H$-biGalois extension of the base field $k$; see \cite{S1}.

Let $R$ and $B$ be right $H$-comodule algebras. Put $L := (R \otimes B)^{co H} = R \Box_H B^{op}$. It is evident that $L$ forms a subalgebra of $R \otimes B^{op}$.
We fix the formal notation
\[ l = \sum l^1 \otimes l^2 \in R \otimes B^{op} \]
for $l \in L$, so that $ll' = \sum l^1{l'}^1 \otimes {l'}^2l^2$.

Extending Schneider's structure theorem for relative Hopf modules, Schauenburg \cite[Theorem 3.3]{S2} described the structure of relative Hopf bimodules as follows. Recall that an $H$-Galois extension $A\subseteq B$ is said to be faithfully flat if $B$ is faithfully flat as a left or right $A$-module.

\begin{thm}\label{ff-Hopf-Galois-ext-Hopf-bimod-cat-str-thm}
Let $H$ be a Hopf algebra and $B$ a right faithfully flat $H$-Galois extension of $A= B^{co H}$. Consider a right $H$-comodule algebra $R$.
	There exists a category equivalence
	\[ F: {_R\mathcal{M}^H_B} \longrightarrow {_L\mathcal{M}}, \;\; M \mapsto M^{co H}, \]
	with quasi-inverse
	\[ G: {_L\mathcal{M}} \longrightarrow {_R\mathcal{M}^H_B}, \;\; N \mapsto N \otimes_A B. \]
For $N \in {_L\mathcal{M}}$, the action $\vartriangleright$ denotes the module structure, with $N$ being a right $A$-module via $na = (1 \otimes a) \vartriangleright n$. The $L$-module structure of $M^{co H}$ is defined by
	\[ l \vartriangleright m = \sum l^1 m l^2.\]
The right $B$-module and $H$-comodule structures on $N \otimes_A B$ are induced by those of $B$, while the left $R$-module structure is given by
\[ r(n \otimes_A b) = \sum \big( r_0 \otimes \kappa^1(r_1) \big) \vartriangleright n \otimes_A \kappa^2(r_1) b.  
\]
\end{thm}

\vspace{0.1in}
\noindent
{\sc III.\ Gorenstein projective modules and dimensions.}
Let $R$ be an associative ring with unit. An acyclic complex of projective $R$-modules
$$\mathbf{P} = \cdots\longrightarrow P_{n+1}\longrightarrow P_{n}\longrightarrow P_{n-1}\longrightarrow\cdots$$
is said to be totally acyclic, provided it remains acyclic after applying ${\rm Hom}_R(-, P)$ for any projective module $P$. A module is Gorenstein projective \cite{EJ} if it is isomorphic to a syzygy of such a totally acyclic complex. 

For any $R$-module $M$, the Gorenstein projective dimension of $M$, denoted by ${\rm Gpd}_R(M)$, is defined by declaring that ${\rm Gpd}_R(M)\leq n$ if and only if $M$ admits a Gorenstein projective resolution of length $n$. The Gorenstein global dimension ${\rm Ggldim}(R)$ of $R$ is defined as the supremum of the Gorenstein projective dimension of all $R$-modules. The terminology is justified by \cite[Theorem 1.1]{BM} which shows that ${\rm Ggldim}(R)$ equals the supremum of the Gorenstein injective dimension of all $R$-modules. Note that ${\rm Ggldim}(R)$ is bounded by the global dimension ${\rm gldim}(R)$ of $R$, and the equality holds if ${\rm gldim}(R)$ is finite. It is well-known that the global dimension is not symmetric, i.e. there are rings whose left and right global dimensions are different. Analogously, in general ${\rm l.Ggldim}(R)$ and ${\rm r.Ggldim}(R)$ are not equal.

\vspace{0.1in}
\noindent
{\sc IV.\ Model categories.}
The notion of a model category was introduced by Quillen \cite{Q} as an axiomatization of homotopy theory; the term ``model category'' is short for ``a category of models for a homotopy theory''. Specifically,  a model category is a category with three specified classes of morphisms, called cofibrations, fibrations and weak equivalences, which satisfy a few axioms which are deliberately reminiscent
of certain properties of continuous maps on topological spaces. We refer to \cite{H1} for details.

Given a model category $\mathcal{C}$, an object $M\in \mathcal{C}$ is called
{\em cofibrant} if $0\rightarrow M$ is a cofibration, it is called {\em fibrant}
if $M\rightarrow 0$ is a fibration, and it is {\em trivial} if $0\rightarrow M$
(equivalently, $M\rightarrow 0$) is a weak equivalence. For any object $X$ in $\mathcal{C}$, by factoring $0\rightarrow X$ as a cofibration $0\rightarrow Q(X)$ follows by a trivial fibration $Q(X)\rightarrow X$, we get a functor $X \mapsto Q(X)$ such that $Q(X)$ is a cofibrant object. We refer to $Q$ as the cofibrant replacement functor of $\mathcal{C}$. 

\vspace{0.1in}

\section{Gorenstein global dimension of Hopf algebras}
\noindent
Note that $k$ is a trivial $H$-module, i.e. $ha=\varepsilon(h)a$ for all $h\in H$
and $a\in k$.  We denote by ${_\varepsilon}k$ (resp. $k_\varepsilon$) the left
(resp. right) trivial module defined by the counit $\varepsilon$ of $H$. In this section, we intend to show that the Gorenstein projective dimension
of $k$ and the Gorenstein global dimension of Hopf algebra $H$ are equal. 

We begin with the following auxiliary simple lemma. We always assume that the antipode of the Hopf algebra is bijective. This assumption is not too restrictive. In fact, for a noetherian Hopf algebra $H$, if either $H$ is semiprime or $H$ is affine PI, then the antipode $S$ is bijective; see \cite[Proposition 1.1]{BZ}. Skryabin \cite{Sk1} conjecture that all noetherian Hopf algebras over a field have bijective antipode. 

\begin{lem}\label{lem:pd}
	For any $H$-modules $M$ and $N$, we have ${\rm pd}_H(M \otimes N) \leq {\rm pd}_H(M)$. Further, if the antipode of $H$ is bijective, then ${\rm pd}_H(M \otimes N) \leq {\rm pd}_H(N)$.
\end{lem}

\begin{proof}
 Since the inequality to be proved is obvious when ${\rm pd}_H(M)$ is infinite, it suffices to assume that ${\rm pd}_H(M) = n$ is finite.
 By applying the functor $- \otimes N$ to a projective resolution of $M$, we obtain a projective resolution of $M \otimes N$ by Lemma \ref{lem:fact}(2).
 The first inequality then follows. Analogously,  Lemma \ref{lem:fact'}(2) yields the second inequality.
\end{proof}

The argument of the following result is standard; see e.g. \cite[Corollary 1.2]{ET1}. The proof is included for completeness.

\begin{lem}\label{lem:monic}
	If ${\rm Gpd}_H({_\varepsilon}k)$ is finite, then there exists an $H$-monomorphism
	$\iota: {_\varepsilon}k\rightarrow \Lambda$ such that ${\rm pd}_H(\Lambda) = {\rm Gpd}_H({_\varepsilon}k)$.
\end{lem}
\begin{proof}
Let ${\rm Gpd}_H({_\varepsilon}k) = n$. If $n=0$, i.e. $k$ is a Gorenstein projective module, then the assertion is clear. Now assume that $n\geq 1$. It follows from \cite[Theorem 2.10]{Hol} that there exists an exact sequence $0\rightarrow N\rightarrow M\rightarrow {_\varepsilon}k\rightarrow 0$, where $M$ is a Gorenstein projective $H$-module, and ${\rm pd}_H(N) = n-1$. For $M$,
there is an exact sequence $0\rightarrow M\rightarrow P\rightarrow L\rightarrow 0$
of $H$-modules such that $L$ is Gorenstein projective and $P$ is projective.
We consider the pushout of $M\rightarrow {_\varepsilon}k$ and $M\rightarrow P$:
	\[\xymatrix{ & & 0\ar[d] & 0\ar[d] \\
		0 \ar[r] &N \ar@{=}[d] \ar[r] & M \ar[d]\ar[r] &{_\varepsilon}k \ar[d]\ar[r] &0 \\
		0 \ar[r] &N \ar[r] & P \ar[r] \ar[d] &\Lambda \ar[r]\ar[d] & 0\\
		&  & L \ar[d] \ar@{=}[r] & L\ar[d]\\
		&  & 0 & 0
	}\]
The middle row implies that ${\rm pd}_H(\Lambda) = {\rm pd}_{H}(N) + 1 = n$,
and then, by the right column we get the required $H$-monomorphism ${_\varepsilon}k\rightarrow \Lambda$.
\end{proof}

\begin{lem}\label{lem:Gproj}
	Assume that there is an $H$-monomorphism $\iota: {_\varepsilon}k\rightarrow \Lambda$
	with ${\rm pd}_H(\Lambda) <\infty$, and let $M$ and $N$ be any $H$-modules.
 \begin{enumerate}
     \item If $M$ is Gorenstein projective, then $M\otimes  N$ is also Gorenstein projective.
     \item Suppose the antipode of $H$ is bijective. If $N$ is Gorenstein projective, then $M\otimes  N$ is also Gorenstein projective.
 \end{enumerate}
\end{lem}

\begin{proof}
(1) Since $M$ is a Gorenstein projective $H$-module, there is a totally acyclic complex of projective $H$-modules ${\bf P}$ such that $M$ is a syzygy of	${\bf P}$. Then, by Lemma \ref{lem:fact}(2) ${\bf P}\otimes N$ is an acyclic complex of projective $H$-modules such that $M\otimes  N$ is a syzygy of ${\bf P}\otimes  N$. For the complex ${\bf P}\otimes  N$, we have 
$Z_i({\bf P}\otimes  N)= Z_i({\bf P})\otimes  N$ for each $i\in \mathbb{Z}$,
	where $Z_i$ denote the $i$-th kernel of the complex. Assume
	${\rm pd}_H(\Lambda) = n$ is finite. By Lemma \ref{lem:pd}, one has ${\rm pd}_H(Z_j({\bf P}) \otimes N \otimes \Lambda)\leq n$ for each $j\in \mathbb{Z}$, and it follows from
	\[0\rightarrow Z_i({\bf P}) \otimes N \otimes \Lambda \rightarrow
	P_i \otimes N \otimes \Lambda \rightarrow \cdots \rightarrow P_{i-n+2} \otimes N \otimes \Lambda \rightarrow	Z_{i-n+1}({\bf P}) \otimes N \otimes \Lambda \rightarrow 0\]
	that $Z_i({\bf P}) \otimes N \otimes \Lambda$ is a projective module
	for each $i\in \mathbb{Z}$. Hence, the complex ${\bf P}\otimes N \otimes \Lambda$ is contractible.
	
	Let $Q$ be any projective $H$-module. Then
	we have an epimorphism $\iota^*: {\rm Hom}_k(\Lambda, Q)\rightarrow Q$
	which is split since $Q$ is projective. In view of Lemma \ref{lem:fact}, there is an isomorphism
	\[{\rm Hom}_H({\bf P}\otimes  N, {\rm Hom}_k(\Lambda, Q))\cong
	{\rm Hom}_H({\bf P} \otimes N \otimes \Lambda, Q). \]
The contractibility of ${\bf P} \otimes N \otimes \Lambda$ implies the acyclicity of ${\rm Hom}_H({\bf P}\otimes N, {\rm Hom}_k(\Lambda, Q))$. Consequently, it follows that ${\rm Hom}_H({\bf P}\otimes N, Q)$ is also acyclic, which means that ${\bf P}\otimes  N$ is a totally acyclic complex. Hence, $M\otimes  N$ is a Gorenstein projective $H$-module, as required.
 
 (2) Assume $N$ is Gorenstein projective and ${\bf T}$ is a totally acyclic complex of projective $H$-modules such that $N$ is a syzygy of ${\bf T}$. Since ${\rm pd}_H(\Lambda)$ is finite, in view of Lemmas  \ref{lem:fact'} and \ref{lem:pd}, an argument analogous to the above will show that $M\otimes {\bf T}$ is a complex of projective modules, and the complex $\Lambda\otimes M\otimes {\bf T}$ is contractible. Since the antipode is bijective, in view of Lemma \ref{lem:fact'}, for any projective $H$-module $Q$ there is an isomorphism
	\[{\rm Hom}_H(M\otimes {\bf T}, {\rm Hom}_k(\Lambda, Q))\cong
	{\rm Hom}_H(\Lambda\otimes M \otimes{\bf T}, Q). \]
Hence, $M\otimes {\bf T}$ is a totally acyclic complex, and $M\otimes  N$ is a Gorenstein projective $H$-module.
\end{proof}

\begin{prop}\label{prop:inequ}
Assume that there is an $H$-monomorphism $\iota: {_\varepsilon}k\rightarrow \Lambda$
with ${\rm pd}_H(\Lambda) <\infty$. If the antipode of $H$ is bijective, then for any $H$-module $M$, we have ${\rm Gpd}_H(M)\leq {\rm pd}_H(\Lambda)$.
\end{prop}

\begin{proof}
Let ${\rm pd}_H(\Lambda) =n$. Since  the antipode of $H$ is bijective, it follows from Lemma \ref{lem:pd} that ${\rm pd}_H(M\otimes \Lambda)\leq n$. Consider the exact sequence of $H$-modules
	\[0\longrightarrow N\longrightarrow P_{n-1}\longrightarrow\cdots\longrightarrow P_1
	\longrightarrow P_0\longrightarrow M\longrightarrow 0\]
where $P_i$ are projective modules. Applying $-\otimes \Lambda$ to this exact sequence shows that $N\otimes \Lambda$ is a projective $H$-module, because each $P_i\otimes \Lambda$ is projective.
	
Now we consider the exact sequence of $H$-modules
$0\rightarrow {_\varepsilon}k\stackrel{\iota}\rightarrow \Lambda\rightarrow \Gamma\rightarrow 0$,
where $\Gamma = {\rm Coker}\,\iota$. By applying
$-\otimes_{k}\Gamma^{\otimes i}$ and $N\otimes  -$ to this exact sequence successively, we have a series of exact sequences
	\[0\longrightarrow N\otimes \Gamma^{\otimes i}\longrightarrow N\otimes  \Lambda\otimes \Gamma^{\otimes i}
	\longrightarrow N\otimes \Gamma^{\otimes i+1}\longrightarrow 0,\]
where $i\geq 0$ are integers with the convention that $\Gamma^{\otimes 0}= {_\varepsilon}k$.
Let $P^i=N\otimes  \Lambda\otimes \Gamma^{\otimes i}$. The projectivity of  $N\otimes  \Lambda$ implies that each $P^i$ is a projective module. We splice these sequences together to obtain a
long exact sequence of $H$-modules
$0\rightarrow N\rightarrow P^0\rightarrow P^1\rightarrow\cdots$.
By pasting this sequence with the projective resolution
$\cdots\rightarrow P_1\rightarrow P_0\rightarrow N\rightarrow 0$ of $N$, we get an acyclic complex ${\bf P}$ of projective $H$-modules, such that
$N$ is a syzygy of ${\bf P}$.
	
Since ${\rm pd}_H(\Lambda) =n$ is finite, analogous to the above argument, it follows that $Z_i({\bf P})\otimes \Lambda$ is a projective $H$-module
for each $i\in \mathbb{Z}$, and the complex ${\bf P}\otimes  \Lambda$ is contractible. Then, for any projective module $Q$, the isomorphism
	\[{\rm Hom}_H({\bf P}, {\rm Hom}_k(\Lambda, Q))\cong
	{\rm Hom}_H({\bf P}\otimes \Lambda, Q) \]
implies that the complex ${\rm Hom}_H({\bf P}, {\rm Hom}_k(\Lambda, Q))$
is acyclic. Note that $Q$ is a direct summand of ${\rm Hom}_k(\Lambda, Q)$. Hence, ${\rm Hom}_H({\bf P}, Q)$ is acyclic, i.e.
${\bf P}$ is a totally acyclic complex, and moreover, $N$ is a Gorenstein projective $H$-module. This completes the proof.
\end{proof}

It is observed by Lorenz and Lorenz  \cite[Section 2.4]{LL} that the global dimension of a Hopf algebra coincides with the projective dimension of its trivial module. In analogy with this, we establish the corresponding equality for Gorenstein global dimension, which follows from Lemma \ref{lem:monic} and Proposition \ref{prop:inequ}. Consequently, the trivial $H$-module $k$ serves as a test module for the Gorenstein projective dimension of all $H$-modules. For the analogous result in the context of group algebras, see \cite[Theorem 1.7]{ET1}.

\begin{thm}\label{thm:fGpd}
	Let $H$ be a Hopf algebra over a field $k$ with a bijective antipode.
  
  (i) The following are equivalent:
	\begin{enumerate}
	\item ${\rm Gpd}_H({_\varepsilon}k)<\infty$.
	\item There is an $H$-monomorphism $\iota: {_\varepsilon}k\rightarrow \Lambda$
	with ${\rm pd}_H(\Lambda)<\infty$.
 \item For any $H$-module $M$, ${\rm Gpd}_H(M)<\infty$.
 \end{enumerate}
\noindent If these equivalent conditions hold, then ${\rm Gpd}_H({_\varepsilon}k)={\rm pd}_H(\Lambda)$, where $\Lambda$ is as in \textup{(2)}.

(ii) ${\rm l.Ggldim}(H) = {\rm Gpd}_H({_\varepsilon}k)$.
\end{thm}

Analogous to the above arguments, we can characterize the finiteness of ${\rm Gpd}_H(k{_\varepsilon})$, and prove the equality ${\rm r.Ggldim}(H)={\rm Gpd}_H(k{_\varepsilon})$. An algebra with finite (left) Gorenstein global dimension is called a (left) Gorenstein algebra; here we do not assume the algebra is noetherian.
\vspace{0.1in}

\noindent In \cite{LZ}, Liu and Zhang proved that an artinian Hopf algebra over a field is finite dimensional; this gives an affirmative answer to a question raised by Bergen. It is possible to further extend this result.

\begin{prop}\label{prop:f.d.H}
Let $H$ be a Hopf algebra over a field $k$ with a bijective antipode. The following
conditions are equivalent:
\begin{enumerate}
	\item $H$ is finite dimensional.
     \item $H$ is left (or right) artinian.
     \item ${\rm Gpd}_H({_\varepsilon}k)= 0$, or ${\rm Gpd}_H(k{_\varepsilon})=0$.
     \item As a left (or right) $H$-module, $k$ is a submodule of a projective module.
\end{enumerate}
\end{prop}

\begin{proof}
The implication (1)$\Rightarrow$(3) holds because a finite dimensional Hopf algebra is Frobenius, and over such an algebra every module is Gorenstein projective. It is immediate from \cite{LZ} that (2)$\Rightarrow$(1). The equivalence of (3) and (4) is clear from Theorem \ref{thm:fGpd}. Now assume ${\rm Gpd}_H({_\varepsilon}k)= 0$. Then Theorem \ref{thm:fGpd} yields ${\rm l.Ggldim}(H)=0$. It is well-known that an algebra whose left (right) Gorenstein global dimension is zero is quasi-Frobenius (see e.g. \cite[Theorem 3.2]{CET}), and any quasi-Frobenius algebra is artinian and self-injective on both sides. Hence, (3)$\Rightarrow$(2) holds. This completes the proof.
\end{proof}

The global dimension of an arbitrary algebra may be asymmetric,  that is, its left and right global dimensions need not coincide. However, for any Hopf algebra $H$ the following equalities hold (see e.g. \cite[Appendix]{WYZ}):
\[{\rm l.gldim}(H)= {\rm pd}_H({_\varepsilon}k)={\rm pd}_{H^e}(H)= {\rm r.gldim}(H)= {\rm pd}_H(k_\varepsilon).\]
For a Hopf algebra, we will show that its left and right Gorenstein global dimensions coincide if the antipode is bijective.

We recall the following result from \cite[Proposition 3.7]{CR}: a left adjoint functor satisfying certain conditions preserves the Gorenstein projective dimension. 

\begin{lem}\label{adj-lem}
Let $\mathcal{A}$ and $\mathcal{B}$ be abelian categories with enough projective objects. Consider an adjoint pair $F: \mathcal{A} \rightleftarrows \mathcal{B}: G$ consisting of exact functors.
If $F$ is faithful and that the class of modules $\mathsf{add}\,G(\mathsf{Proj}(\mathcal{B}))$ coincides with $\mathsf{Proj}(\mathcal{A})$, then we have ${\rm Gpd}_{\mathcal{A}}(X) = {\rm Gpd}_{\mathcal{B}}(F(X))$ for each object $X \in \mathcal{A}$.
\end{lem}

The canonical algebra map $\varphi: H \to H^e, h \mapsto \sum h_1 \otimes Sh_2$ induces a functor $G: {_H\!\M_H} \to {_H\!\M}$, which has a left adjoint $F = - \otimes H: {_H\!\M} \to {_H\!\M_H}$.
For any left $H$-module $N$, the $H$-$H$-bimodule structure on $F(N) = N \otimes H$ is given by $h'(n \otimes h)h'' = \sum h'_1 \cdot n \otimes h'_2 h h''$.
Note that $(\id_H \otimes \varepsilon) \varphi = \id_H$, so $H$ is a direct summand of $H^e$ as an $H$-module.

\begin{prop}\label{prop:equ-1}
Let $H$ be a Hopf algebra over a field $k$.
For any left $H$-module $M$, we have ${\rm Gpd}_H(M) = {\rm Gpd}_{H^e}(M \otimes H)$.
\end{prop}

\begin{proof}
It suffices to examine that the adjoint pair of functors $F: {_H\!\M} \rightleftarrows {_H\!\M_H}: G$ satisfies the conditions of Lemma \ref{adj-lem}.
It is evident that the functor $F$ is exact and faithful, while the functor $G$ is exact.
The equality $\mathsf{add} G(\mathsf{Proj}({_H\!\M_H})) = \mathsf{Proj}({_H\!\M})$ follows directly from two key facts:
$H^e$ is a projective $H$-module;
$H$ is direct summand of $H^e$ as an $H$-module.
\end{proof}

Note that $k \otimes H \cong H$ as $H^e$-modules.
Specifying the above to the regular $H$-bimodule $H$, we have ${\rm Gpd}_H({_\varepsilon}k) = {\rm Gpd}_{H^e}(H)$.
Combining the preceding equality with Theorem \ref{thm:fGpd} and its dual formulation for right $H$-modules, we obtain the following result.

\begin{cor}\label{cor:equ-2}
Let $H$ be a Hopf algebra over a field $k$ with a bijective antipode. Then
   \[{\rm l.Ggldim}(H)= {\rm Gpd}_H({_\varepsilon}k) ={\rm Gpd}_{H^e}(H)= {\rm r.Ggldim}(H)= {\rm Gpd}_H(k_\varepsilon).\]
\end{cor}

\begin{rk}
(1) Note that any Hopf algebra $H$ naturally forms an $H$-$H$-bi-Galois object.
By applying Lemma \ref{adj-lem}, we further obtain a generalization of Proposition \ref{prop:equ-1} to faithfully flat Hopf Galois extensions, as stated in Proposition \ref{Gpd-L-RB-prop}.

(2) Note that ${\rm pd}_{H^e}(H)$ coincides with the Hochschild cohomological dimension 
${\rm Hdim}(H)$ of $H$, which is defined by
\[{\rm Hdim}(A) = \sup \{ n \mid {\rm HH}^n(H, M) \neq 0 \text{ for some } H\text{-bimodule } M \}. \]
In analogy with the Hochschild cohomological dimension, for any algebra $A$ we define the generalized Hochschild cohomological dimension
\[\underline{{\rm Hdim}}(A) = \sup \{ n \mid {\rm HH}^n(A, P) \neq 0 \text{ for some projective } A\text{-bimodule } P \}. \]
By \cite[Theorem 2.20]{Hol} we have $\underline{{\rm Hdim}}(H)\leq {\rm Gpd}_{H^e}(H)$; moreover, if ${\rm Gpd}_{H^e}(H)$ is finite, then the equality ${\rm Gpd}_{H^e}(H)=\underline{{\rm Hdim}}(H)$ holds. However, this equality may fail when ${\rm Gpd}_{H^e}(H)$ is infinite.
\end{rk}

In connection with the existence of complete cohomological functors in the category of modules, Gedrich and Gruenberg \cite{GG} introduced two invariant for any ring $R$: the supremum of the injective lengths of projective left $R$-modules ${\rm l.silp}(R)$, and the supremum of the projective lengths of injective left R-modules ${\rm l.spli}(R)$. In the same way, we may consider right $R$-modules and define the right invariants ${\rm r.silp}(R)$ and ${\rm r.spli}(R)$.

Gedrich and Gruenberg observed that the relation between ${\rm l.silp}(R)$ and ${\rm l.spli}(R)$ is unclear for a general ring $R$; however, if both invariants are finite, then ${\rm l.silp}(R) = {\rm l.spli}(R)$; see \cite[1.6]{GG}.
In the special case where $R$ is an Artin algebra, the equality  ${\rm l.silp}(R) = {\rm l.spli}(R)$ is equivalent to Gorenstein Symmetry Conjecture, a long-standing conjecture in representation theory.  For Hopf algebras, the following result is \cite[Theorem 2.4]{GG}.

\begin{lem}\label{silp=<spli-Hopf-lem}
Let $H$ be a Hopf algebra over a field $k$.
Then ${\rm l.silp}(H) \leq {\rm l.spli}(H)$.
\end{lem}

Recall that a ring $R$ is left $\aleph_0$-noetherian if every left ideal of $R$ is countably generated. In particular, a $k$-algebra generated by a countably (or finitely) generated linear space is always $\aleph_0$-noetherian.

\begin{prop}\label{silp=spli-prop}
Let $H$ be a Hopf algebra over a field $k$ with a bijective antipode.
If $H$ is $\aleph_0$-noetherian, then ${\rm l.silp}(H) = {\rm l.spli}(H)$.
\end{prop}
\begin{proof}
Since the antipode $S$ is bijective, it yields an isomorphism $H\cong H^{op}$.
By \cite[Corollary 3.7]{E1} we have ${\rm l.spli}(H) \leq {\rm l.silp}(H)$.
Together with Lemma \ref{silp=<spli-Hopf-lem}, the result follows.
\end{proof}

Emmanouil \cite{E1} proved the equality ${\rm l.spli}(\mathbb{Z}G) = {\rm l.silp}(\mathbb{Z}G)$ for any group $G$. We extend this result to pointed Hopf algebras by applying the following observation, implicit in the proof of \cite[Proposition 4.2]{E1}.

\begin{lem}\label{spli=<silp-Hopf-lem}
Let $R$ be a ring.
Suppose that the following conditions hold.
\begin{enumerate}
	\item $R$ is a filtered colimit $\varinjlim R_{\lambda}$ of some subalgebras $\{ R_{\lambda} \}_{\lambda \in \Lambda}$ of $R$.
	\item $R_{\lambda}$ is isomorphic to its opposite ring $R_\lambda^{op}$ for any $\lambda \in \Lambda$.
	\item $R_{\lambda}$ is a $\aleph_0$-noetherian ring.
	\item $R$ is flat as a right $R_{\lambda}$-module, and $R_{\lambda}$ is a direct summand of $R$ as a left $R_{\lambda}$-module.
\end{enumerate}
Then ${\rm l.spli}(R) \leq {\rm l.silp}(R)$.
\end{lem}

\begin{prop}\label{prop:silp=spli+1}
Let $H$ be a Hopf algebra over a field $k$ with a bijective antipode.
Suppose that $H$ is pointed, or the coradical of $H$ is countably dimensional.
Then ${\rm l.silp}(H) = {\rm l.spli}(H)$.
\end{prop}
\begin{proof}
By Lemmas \ref{silp=<spli-Hopf-lem} and \ref{spli=<silp-Hopf-lem}, it suffices to verify that $H$ satisfies the conditions in Lemma \ref{spli=<silp-Hopf-lem}.
	
(1)  Let $\{ V_{\lambda} \}_{\lambda \in \Lambda}$ be defined as follows: If $H$ is pointed, let $\{ V_{\lambda} \}_{\lambda \in \Lambda}$ consist of all finite dimensional subcoalgebras of $H$; if the coradical $H_0$ is countably dimensional, let $\{ V_{\lambda} \}_{\lambda \in \Lambda}$ consist of all countably dimensional subcoalgebras containing $H_0$.	For each $\lambda$, let $H_{\lambda}$ denote the subalgebra of $H$ generated by $\bigcup_{i \in \mathbb{Z}} S^i(V_{\lambda})$.
For any $x \in H$, there exists a finite dimensional subcoalgebra $V \subseteq H$ such that $x \in V$. Thus, $H$ is the filtered colimit $\varinjlim_{\lambda \in \Lambda} H_{\lambda}$ of the subalgebras $\{ H_{\lambda} \}_{\lambda \in \Lambda}$.
	
(2)	By construction, each $H_{\lambda}$ is a Hopf subalgebra of $H$, and the restriction $S|_{H_{\lambda}}$ is a bijective antipode.
Consequently, $S|_{H_{\lambda}}$ is an algebra isomorphism from $H_{\lambda}$ onto $H_{\lambda}^{\mathrm{op}}$.

(3)	Since $H_{\lambda}$ is generated, as a $k$-algebra, by a countably dimensional linear space $\bigcup_{i \in \mathbb{Z}} S^i(V_{\lambda})$, it follows that $H_{\lambda}$ is $\aleph_0$-noetherian.
	
(4)	By \cite[Theorem 2.1]{MW} (see also \cite[Theorem 1.2]{MS}), if $H$ is faithfully flat as a left $H_{\lambda}$-module, then $H_{\lambda}$ is a direct summand of $H$ as a left $H_{\lambda}$-module.
Since the antipode $S$ is bijective, the opposite algebra $H^{\mathrm{op}}$ is also a Hopf algebra.
Applying the same result to $H^{\mathrm{op}}$ and its Hopf subalgebra $H_{\lambda}^{\mathrm{op}}$, we find that if $H$ is faithfully flat as a right $H_{\lambda}$-module, then $H_{\lambda}$ is a direct summand of $H$ as a right $H_{\lambda}$-module.
Therefore, it suffices to show that $H$ is faithfully flat as both a left and a right $H_{\lambda}$-module.
If $H$ is pointed, this follows from \cite{R2}.
If the coradical $H_0$ is countably dimensional instead, then the construction in (1) yields $H_0 \subseteq H_{\lambda}$, and hence $H$ is faithfully flat over $H_{\lambda}$ by \cite[Corollary 2.3]{R1}.
\end{proof}

It is well-known that ${\rm l.Ggldim}(R) = \max\{ {\rm l.silp}(R), {\rm l.spli}(R) \}$; see e.g. \cite[Theorem 4.1]{E2}. Moreover, if $R$ is a noetherian ring, then ${\rm l.silp}(R)$ equals ${\rm id}_R(R)$, the injective dimension of the left regular $R$-module $R$.

\begin{cor}\label{silp=spli-noe-Hopf-cor}
Let $H$ be a noetherian Hopf algebra over a field $k$ with a bijective antipode.
Then ${\rm l.Ggldim}(H) = {\rm l.silp}(H) = {\rm l.spli}(H) = {\rm id}_H(H)$.
\end{cor}
\begin{proof}
This follows from Proposition \ref{silp=spli-prop} together with the identities ${\rm l.silp}(H) = {\rm id}_H(H)$ and ${\rm l.Ggldim}(H) = \max\{ {\rm l.silp}(H), {\rm l.spli}(H) \}$.
\end{proof}

We conclude this section by examining the behavior of the Gorenstein projective dimension in the context of left coideal subalgebras and Hopf quotient algebras.
In the following, let $H$ be a Hopf algebra with a bijective antipode and $K$ be a left coideal subalgebra of $H$.

\begin{cor}\label{cor:Gpd-subalg}
If $H$ is projective as a left $K$-module, then ${\rm Gpd}_K({_\varepsilon}k)\leq {\rm Gpd}_H({_\varepsilon}k)$.
\end{cor}
\begin{proof}
The desired inequality is immediate if	${\rm Gpd}_H({_\varepsilon}k)=\infty$. We may therefore assume ${\rm Gpd}_H({_\varepsilon}k)=n$ is finite. Then Theorem \ref{thm:fGpd} yields an $H$-monomorphism $\iota: {_\varepsilon}k\rightarrow \Lambda$ with ${\rm pd}_H(\Lambda)=n$. It is clear that the restriction functor from the category of $H$-modules to the category of $K$-modules preserves projective objects precisely when $H$ is projective as a left $K$-module. Under this hypothesis, restricting $\iota$ gives a $K$-monomorphism $\iota': {_\varepsilon}k\rightarrow \Lambda$ with ${\rm pd}_K(\Lambda)\leq n$. Hence ${\rm Gpd}_K({_\varepsilon}k)\leq n$, as required.
\end{proof}

A left coideal subalgebra $K$ of $H$ is said to be normal if
\[
\operatorname{ad}_l(h)(x) := \sum h_{1} x (Sh_{2})
\quad \text{and} \quad
\operatorname{ad}_r(h)(x) := \sum (Sh_{1}) x h_{2}
\]
both lie in $K$ for all $x \in K$ and $h \in H$.
Let $K^+ = \operatorname{Ker}\,\varepsilon \cap K$.
It is readily verified that $HK^+ = K^+H$. Since $S(HK^+) = K^+H$ by \cite[Lemma 1.4]{MS}, it follows that $HK^+$ is a Hopf ideal of $H$.
We denote the quotient Hopf algebra $H/HK^+$ by $\overline{H}$ and let $\pi : H \to \overline{H}, h \mapsto \overline{h}$ be the canonical quotient map.
As $S(HK^+) = HK^+$, the induced antipode of $\overline{H}$ is also bijective.
Then $H$ becomes a right $\overline{H}$-comodule algebra via $h \mapsto \sum h_{1} \otimes \overline{h_{2}}$.
If $H$ is a faithfully flat left $K$-module, then $H^{co \overline{H}} = K$; see \cite[Theorem 2.1]{MW}.
In this case, $K \subseteq H$ is an $\overline{H}$-Galois extension; that is, the Galois map
\[
H \otimes_K H \longrightarrow H \otimes \overline{H},
\qquad h' \otimes h \mapsto \sum h' h_{1} \otimes \overline{h_{2}}
\]
is bijective, with inverse given by $h' \otimes \overline{h} \mapsto \sum h' Sh_{1} \otimes h_{2}$.

For any $\overline{H}$-module $M$ and $H$-module $N$, the $K$-invariant submodule $\Hom_K(k, N)$ of $N$ admits a natural left $\overline{H}$-module structure defined by
\begin{equation}\label{mod-str-on-Hom_K(k, N)-eq}
  \overline{h} \rightharpoonup x = hx,
\end{equation}
where $x \in \Hom_K(k, N)$ and $h \in H$ is an arbitrary lift of $\overline{h} \in \overline{H}$.
This action is well-defined since $K^{+}$ annihilates $\Hom_K(k, N)$. 
Consequently, one obtains a natural isomorphism
\[
\Hom_H(M, N) \cong \Hom_{\overline{H}}\bigl(M, \Hom_K(k, N)\bigr).
\]
Thus $\Hom_K(k, -)$ is right adjoint to the exact forgetful functor ${}_{\overline{H}}\mathcal{M} \to {}_H\mathcal{M}, M \mapsto M$.
Hence it preserves injective objects; i.e., injective $H$-modules are sent to injective $\overline{H}$-modules. Furthermore, as $H$ is flat over $K$ on the right, the restriction of scalars ${}_H\mathcal{M} \to {}_K\mathcal{M}$ also preserves injectives.
By \cite[Corollary~10.8.3]{W}, there is a quasi-isomorphism of complexes
\begin{equation}\label{Hopf-ext-spectral-sequence}
\RHom_H(M, N) \cong \RHom_{\overline{H}}(M, \RHom_K(k, N) ),
\end{equation}
which may be regarded as a special case of the Stefan spectral sequence \cite{Ste}; see also \cite{LeC}.

The proof of the following lemma appears in \cite[Proposition 5.6]{LeC} in the case where $K$ is a Hopf subalgebra. For the convenience of the reader, we include a brief sketch.

\begin{lem}\label{lem:LeC}
 Assume that $K$ is a normal left coideal subalgebra of $H$ and that $H$ is a faithfully flat left $K$-module. If $K$ is of type ${\rm FP}_{\infty}$ (that is, $_Kk$ admits a projective resolution by finitely generated projective $K$-modules), then
  \[ \RHom_K(k, H) \cong \RHom_K(k, K) \otimes \overline{H}, \]
where the $\overline{H}$-module structure on $\RHom_K(k, K) \otimes \overline{H}$ is given by multiplication on the second tensor factor.   
\end{lem}

\begin{proof}
We first recall the module structures that will be used in the proof. For left $H$-modules $M$ and $N$, the space $\Hom_K(M,N)$ is naturally a left $\overline{H}$-module via
\[ (\overline{h}\rightharpoonup g)(x) = \sum h_1 g((Sh_2)x),\qquad g\in \Hom_K(M,N),\ h\in H,\ x\in M. \]
When $M=k$, this $\overline{H}$-module structure agrees with the one defined in \eqref{mod-str-on-Hom_K(k, N)-eq}. The bifunctor
 $\Hom_K(-,-):{}_H\mathcal{M}\times {}_H\mathcal{M}
\To {}_{\overline{H}}\mathcal{M}$
induces a derived bifunctor
$\RHom_K(-,-): \D^{-}({}_H\mathcal{M})\times\D^{+}({}_H\mathcal{M}) \To \D^{+}({}_{\overline{H}}\mathcal{M})$.
In particular, $\RHom_K(k,-)$ agrees with the one appearing in \eqref{Hopf-ext-spectral-sequence}.

We shall also use the following compatible module structures.  The space $\Hom_K(M,K)$ is a right $H$-module under
\[(f\leftharpoonup h)(x)=\sum (Sh_2)f((S^2h_1)x)h_3,\qquad f\in\Hom_K(M,K),\ h\in H,\ x\in M. \]
If $M'$ is a right $H$-module, then $M'\otimes_K N$ is a left $\overline{H}$-module via 
\[ \overline{h}\rightharpoonup (x\otimes_K y)
=\sum xS^{-1}h_2\otimes_K h_1y. \]
In particular, there is an isomorphism of left $\overline{H}$-modules $M'\otimes_K H \To M'\otimes \overline{H}$ given by $x\otimes_K h \mapsto \sum xh_2\otimes \overline{h_1}$,
where $\overline{H}$ acts on the second tensor factor on the right-hand side.

Now consider the natural map 
\[
\Hom_K(M,K)\otimes_K N
\To \Hom_K(M,N), \qquad
f\otimes_K y \longmapsto \bigl(x\mapsto f(x)y\bigr).
\]
With the module structures above, this map is $\overline{H}$-linear.
Moreover, it is an isomorphism whenever $M$ is a finitely generated
projective $K$-module.

Since $K$ is of type ${\rm FP}_{\infty}$, choose a projective resolution
$P_\bullet\to k$ in which every $P_i$ is a finitely generated projective $K$-module.  Applying the preceding isomorphism degreewise, with $M=P_i$ and $N=H$, gives an isomorphism of complexes of left
$\overline{H}$-modules
\[
\Hom_K(P_\bullet,H)
\cong
\Hom_K(P_\bullet,K)\otimes_K H.
\]
Consequently, $\RHom_K(k,H)
\cong
\RHom_K(k,K)\otimes^{\textbf{L}}_K H$
in $\D({}_{\overline{H}}\mathcal{M})$.

Finally, since $H$ is faithfully flat, and hence flat, as a left $K$-module, the derived tensor product with $H$ may be computed by the ordinary tensor product.  Applying the above isomorphism $M'\otimes_K H\cong M'\otimes\overline{H}$
degreewise to a complex representing $\RHom_K(k,K)$ yields
\[
\RHom_K(k,K) \otimes^{\textbf{L}}_K H
\cong
\RHom_K(k,K)\otimes\overline{H}.
\]
Here $\overline{H}$ acts by multiplication on the second tensor factor.
Combining the two isomorphisms gives $\RHom_K(k,H)
\cong
\RHom_K(k,K)\otimes\overline{H}$, as required.
\end{proof}

The proof of the following result draws on techniques developed for Gorenstein cohomological dimensions of groups in \cite[Proposition 2.9]{ET1} and \cite[Corollary 4.7]{ET3}, though our approach differs in several key respects.

\begin{cor}\label{cor:Gpd-quoalg}
Assume that $K$ is a normal left coideal subalgebra of $H$ and that $H$ is a faithfully flat left $K$-module.
\begin{enumerate}
  \item ${\rm Gpd}_H({_\varepsilon}k) \leq {\rm Gpd}_K({_\varepsilon}k) + {\rm Gpd}_{\overline{H}}({_\varepsilon}k)$. In particular, if $K$ is also a Hopf subalgebra of $H$, then ${\rm l.Ggldim}(H) \leq {\rm l.Ggldim}(K) + {\rm l.Ggldim}(\overline{H})$.
  \item If $K$ is of type ${\rm FP}_{\infty}$ and ${\rm l.Ggldim}(\overline{H}) < \infty$, then 
  \[{\rm l.Ggldim}(H) = {\rm Gpd}_K({_\varepsilon}k) + {\rm l.Ggldim}(\overline{H}).\]
\end{enumerate}
\end{cor}

\begin{proof}
(1) By Theorem \ref{ffHGe-thm} and \cite[Theorem 1.2]{MS}, $H_K$ is faithfully flat, $_KH$ is projective, and $K$ is a direct summand of $H$ as a left $K$-module.
If either ${\rm Gpd}_K(k)$ or ${\rm Gpd}_{\overline{H}}(k)$ is infinite, the result is trivial.
Thus, we may assume that both ${\rm Gpd}_K(k) = n$ and ${\rm Gpd}_{\overline{H}}(k) = m$ are finite.
By Theorem \ref{thm:fGpd}, there exists an $\overline{H}$-monomorphism $_{\varepsilon}k \to \Lambda$ with ${\rm pd}_{\overline{H}}(\Lambda) = m$.
Since
\[ {\rm Gpd}_K({_\varepsilon}k) = {\rm Gpd}_H(H \otimes_K {_\varepsilon}k) = {\rm Gpd}_H(\overline{H}) \]
by applying Lemma \ref{adj-lem}, it follows that
\[ {\rm Gpd}_{H}(\Lambda) \leq {\rm pd}_{\overline{H}}(\Lambda) + {\rm Gpd}_H(\overline{H}) = m + n. \]
By \cite[Proposition 1.2]{ET}, there exists an $H$-monomorphism $\Lambda \to \Gamma$ such that ${\rm pd}_{H}(\Gamma) = {\rm Gpd}_{H}(\Lambda) \leq m+n$. Composing this with the $H$-monomorphism $_{\varepsilon}k \to \Lambda$ yields an $H$-monomorphism $_{\varepsilon}k \to \Gamma$. Applying Theorem \ref{thm:fGpd} again, we obtain ${\rm pd}_{H}(\Gamma) = {\rm Gpd}_{H}(\Lambda) \leq m+n$.

Assume that $K$ is a Hopf subalgebra.
By Theorem \ref{thm:fGpd}, we only need to show that the antipode of $K$ is bijective.
For any $x \in K$, the equality 
\[ad_l(S^{-1}x_2)(Sx_1) = \sum (S^{-1}x_3) (Sx_1) x_2 = S^{-1}x\] 
implies $S^{-1}K \subseteq K$.
Consequently, the restriction $S|_K$ is surjective. Since $S$ is injective, so is $S|_K$. Thus $K$ admits a bijective antipode, as needed.

(2) Now suppose further that ${\rm l.Ggldim}(\overline{H}) = {\rm r.Ggldim}(\overline{H}) = m < \infty$ and that ${\rm Gpd}_{K}(k) = n$.
By \cite[Theorem 2.20]{Hol}, $ m = \sup\{ i \mid \Ext^i_{\overline{H}}(k, P) \neq 0 \text{ for some projective $\overline{H}$-module $P$} \}$.
In particular, there exists a cardinal $\kappa$ such that $\Ext^m_{\overline{H}}(k, -)$ does not vanish on the free $\overline{H}$-module $\overline{H}^{(\kappa)}$.
In what follows, we prove that $\Ext^{n+m}_H({}_{\varepsilon}k, H^{(\kappa)}) \neq 0$.

Since $K$ is of type ${\rm FP}_{\infty}$, the functor $\Ext^q_K(k, -)$ commutes with direct sums; hence we have $\Ext^q_K(k, H^{(\kappa)}) \cong \Ext^q_K(k, H)^{(\kappa)}$.
By Lemma \ref{lem:LeC}, $\Ext_K^q(k, H^{(\kappa)})\cong \Ext_K^q(k, K) \otimes \overline{H}^{(\kappa)}$ as left $\overline{H}$-modules.
Since ${\rm Gpd}_K(k) = n < \infty$, it follows that $\Ext^n_K(k, K) \neq 0$ and $\Ext^q_K(k, K) = 0$ for any $q > n$.
Thus $\overline{H}^{(\kappa)}$ is a direct summand of $\Ext_K^n(k, H^{(\kappa)})$ as an $\overline{H}$-module, and each $\Ext_K^q(k, H^{(\kappa)})$ is a free $\overline{H}$-module (possibly zero) for any $q$.
Since ${\rm Gpd}_{\overline{H}}({}_{\varepsilon}k) = m < \infty$, we have $\Ext^p_{\overline{H}^{op}}(k, \Ext^q_K(k, H^{(\kappa)})) = 0$ for any $p > m$.
Take $M = k$ and $N = H^{(\kappa)}$ in \eqref{Hopf-ext-spectral-sequence}.
Because $\Ext^q_K(k, H^{(\kappa)}) = 0$ whenever $q > n$, the spectral sequence associated to the isomorphism \eqref{Hopf-ext-spectral-sequence} yields
\[\Ext^{n+m}_H(k, H^{(\kappa)}) \cong \Ext^m_{\overline{H}}(k, \Ext^n_K(k, H^{(\kappa)})).\]
By our choice of $\kappa$, we have $\Ext^m_{\overline{H}}(k, \overline{H}^{(\kappa)}) \neq 0$.
Since $\overline{H}^{(\kappa)}$ is a direct summand of $\Ext^n_K(k, H^{(\kappa)})$, it follows that $\Ext^{n+m}_H(k, H^{(\kappa)}) \neq 0$.
This implies $ {\rm l.Ggldim}(H) \geq n + m$.
Combining this with the earlier upper bound ${\rm l.Ggldim}(H) \leq n + m$, we conclude the required equality ${\rm l.Ggldim}(H) = {\rm Gpd}_K({_\varepsilon}k) + {\rm l.Ggldim}(\overline{H})$.
\end{proof}

\vspace{0.1in}

\section{Gorenstein invariants under monoidal Morita-Takeuchi equivalences}
\noindent In this section, we intend to show that Gorenstein global dimension, self-injective dimension, and AS Gorenstein property are invariant under  monoidal Morita-Takeuchi equivalence of Hopf algebras.

Throughout this section, $H$ is assumed to be a Hopf algebra with a bijective antipode.
Let $A\subseteq B$ be a faithfully flat $H$-Galois extension and $R$ a right $H$-comodule algebra. Recall that $L:= (R \otimes B)^{co H} = R \Box_H B^{op}$ is a subalgebra of $R \otimes B^{op}$.
By Theorem \ref{ff-Hopf-Galois-ext-Hopf-bimod-cat-str-thm}, there is an equivalence 
\[ \xymatrix{ (-)^{co H}: {_R\!\M_B^H}\ar@<0.5ex>[r] &{_L\!\M}: - \otimes_A B \ar@<0.5ex>[l] } \]
of abelian categories.
Therefore, the category ${_R\!\M_B^H}$ has enough projective objects, as this is already true for ${_L\!\M}$.

Let $U: {_R\!\M_B^H} \to {_R\!\M_B}$ be the forgetful functor, which acts by discarding the $H$-comodule structure of any Hopf module in ${_R\!\M_B^H}$.
For any $R \otimes B^{op}$-module $M$, $M \otimes H$ can be viewed as a Hopf module in ${_R\!\M_B^H}$ through the structure defined by
\begin{equation*}
	r(m \otimes h)b = \sum r_0 m b_0 \otimes r_1 h b_1, \qquad \rho(m \otimes h) = \sum m \otimes h_1 \otimes h_2,
\end{equation*}
for any $m \in M$, $h \in H$, $r \in R$ and $b \in B$.
The functor $- \otimes H: {_R\!\M_B} \to {_R\!\M_B^H}$ serves as a right adjoint to the forgetful functor $U$, that is, there exists a natural isomorphism
\begin{equation*}
	\Hom_{_R\!\M_B}(U(X), M) \To \Hom_{_R\!\M^H_B}(X, M \otimes H), \;\; f \mapsto \big( x \mapsto \sum f(x_0) \otimes x_1 \big)
\end{equation*}
for any $X \in {_R\!\M_B^H}$ and $M \in {_R\!\M_B}$.
The inverse map is given by $(\id_M \otimes \varepsilon) \circ g \mapsfrom g$.

Consider the composition functors given by the following commutative diagrams:
\[ \xymatrix{ {_R\!\M_B^H} \ar[r]^-{U} & {_R\!\M_B} \\ {_L\!\M} \ar[u]^{- \otimes_A B} \ar[ru]_-{F} } \qquad \xymatrix{ {_R\!\M_B^H} \ar[d]_{(-)^{co H}} & {_R\!\M_B} \ar[l]_-{- \otimes H} \ar[dl]^-{G} \\ {_L\!\M} }. \]
One checks directly that $F: {_L\!\M} \rightleftarrows {_R\!\M_B}: G$ forms an adjoint pair, and that $G: {_R\!\M_B} \to {_L\!\M}$ coincides with the restriction functor.

\begin{prop}\label{Gpd-L-RB-prop}
For any left $L$-module $M$, the following equalities hold:
  \[ {\rm Gpd}_{_R\!\M^H_B}(M \otimes_A B) = {\rm Gpd}_L(M) = {\rm Gpd}_{R \otimes B^{op}}(F(M)). \]
In particular, $M$ is Gorenstein projective if and only if $M \otimes_A B$ is also Gorenstein projective in $_R\!\M^H_B$.
\end{prop}

\begin{proof}
It suffices to examine that the adjoint pair $F: {_L\!\M} \rightleftarrows {_R\!\M_B}: G$ satisfies the conditions of Lemma \ref{adj-lem}. Because $B$ is a faithfully flat $A$-module, the functor $- \otimes_A B$ is faithful.
The forgetful functor $U$ is clearly faithful as well, so their composite $F$ is also faithful. The equality $\mathsf{add} G(\mathsf{Proj}({_R\!\M_B})) = \mathsf{Proj}({_L\!\M})$ follows directly from two facts:
$G(R \otimes B^{op})$ is a projective $L$-module, and $L$ is direct summand of $G(R \otimes B^{op})$ as an $L$-module; see \cite[Proposition 3.9]{Zh}.
\end{proof}

In the following, we specify the right $H$-comodule algebra $R$ to $B$.
Indeed, ${_B\!\M^H_B}$ is a monoidal category equipped with the tensor product over $B$.
Although $L = (B^e)^{co H}$ is not generally a bialgebra, the category ${_L\!\M}$ nevertheless admits a monoidal structure, which is induced by that of ${_B\!\M^H_B}$.
In fact, $L$ is a bialgebroid (as defined in \cite{S2}), which generalizes the notion of a bialgebra to a non-commutative base ring. Specifically, $(B^e)^{co H}$ is known as the Ehresmann-Schauenburg bialgebroid \cite[34.14]{BW}.

Moreover, we assume $A = k$. In this setting, as shown in \cite{S1}, $L= (B^e)^{co H}$ inherits a Hopf algebra structure. Its comultiplication is given by
\[ \Delta(l) = \sum {l^1}_0 \otimes \kappa^1({l^1}_1) \otimes \kappa^2({l^1}_1) \otimes l^2\]
for $l = \sum l^1 \otimes l^2 \in L \subseteq B^e$.
The counit of $L$ is defined as $\varepsilon(l) = \sum l^1l^2$, and the antipode is given by $S(l) = \sum {l^2}_0 \otimes \kappa^1({l^2}_1) l^1 \kappa^2({l^2}_1)$.
In addition, $B$ is endowed with a left $L$-comodule algebra structure via the comodule map $\delta: B \rightarrow L \otimes B$ defined as
\[b \mapsto \sum \big( b_0 \otimes \kappa^1(b_1) \big) \otimes \kappa^2(b_1).\]
With these structures, $B$ is an $L$-$H$-biGalois extension of $k$;  that is, it is simultaneously a right $H$-Galois extension and a left $L$-Galois extension of $k$, and the two comodule structures make it an $L$-$H$-bicomodule.

Schauenburg has shown in \cite[Corollary 5.7]{S1} the equivalence of the two assertions: Two Hopf algebras $L$ and $H$ are monoidally Morita-Takeuchi equivalent (i.e. their comodule categories are monoidally equivalent), if and only if there exists an $L$-$H$-biGalois extension of $k$.

\begin{thm}\label{mMTe-thm}
Let $H$ and $L$ be monoidally Morita-Takeuchi equivalent Hopf algebras with bijective antipodes, and let $B$ be an $L$-$H$-biGalois extension of $k$. Then
\[ {\rm Ggldim}(H) = {\rm Gpd}_H(k) = {\rm Gpd}_{B^e}(B) = {\rm Gpd}_L(k) = {\rm Ggldim}(L). \]
\end{thm}

\begin{proof}
We apply Proposition~\ref{Gpd-L-RB-prop} twice, using the two
Galois structures on the biGalois object $B$.
First, regard $B$ as a right $H$-Galois object. Taking $A=k$ and
$R=B$ in Proposition~\ref{Gpd-L-RB-prop}, the corresponding Hopf
algebra is
$(B^e)^{co H}\cong L$.
For the trivial $L$-module $k$, the functor $F$ occurring in
Proposition~\ref{Gpd-L-RB-prop} sends $k$ to the underlying
$B$-bimodule of $k\otimes_k B$, which is naturally isomorphic to
$B$. Hence
${\rm Gpd}_L(k)={\rm Gpd}_{B^e}(F(k))={\rm Gpd}_{B^e}(B)$.

Next, denote by $H^{\mathrm{cop}}$ and $L^{\mathrm{cop}}$ the coopposite Hopf algebras of $H$ and $L$, respectively.
The right $H$-comodule algebra structure on $B$ induces a natural left $H^{\mathrm{cop}}$-comodule algebra structure on $B$, while the left $L$-comodule algebra structure on $B$ induces a natural right $L^{\mathrm{cop}}$-comodule algebra structure on $B$.
Consequently, the $L$--$H$-biGalois structure on $B$ can be naturally regarded as an $H^{\mathrm{cop}}$--$L^{\mathrm{cop}}$-biGalois structure.
Under this identification,
\[
	(B^e)^{co L^{\mathrm{cop}}} \cong H^{\mathrm{cop}}
\]
as Hopf algebras.
Applying Proposition~\ref{Gpd-L-RB-prop} again, now to the right
$L^{\mathrm{cop}}$-Galois object $B$, with $A=k$, $R=B$, and
$M=k$, gives 
${\rm Gpd}_H(k)= {\rm Gpd}_{B^e}(B)$.
Thus the two applications of Proposition~\ref{Gpd-L-RB-prop} have
the same intermediate invariant, and therefore
${\rm Gpd}_H(k)={\rm Gpd}_{B^e}(B)={\rm Gpd}_L(k)$.

Finally, since both $H$ and $L$ have bijective antipodes,
Theorem~\ref{thm:fGpd} identifies the Gorenstein global dimension
of each Hopf algebra with the Gorenstein projective dimension of
its trivial module:
${\rm Ggldim}(H)={\rm Gpd}_H(k)$ and ${\rm Ggldim}(L)={\rm Gpd}_L(k)$.
Combining these equalities yields
\[ {\rm Ggldim}(H)= {\rm Gpd}_H(k)
={\rm Gpd}_{B^e}(B) = {\rm Gpd}_L(k)
= {\rm Ggldim}(L), \]
as required.
\end{proof}


The above result implies that monoidally Morita-Takeuchi equivalence preserves Gorensteinness of Hopf algebras. We expect to know whether it also preserves AS Gorenstein property of Hopf algebras. It is well-known that any finite dimensional Hopf algebra is Gorenstein, and AS Gorenstein as well. It follows from Wu and Zhang \cite{WZ} that every noetherian affine PI Hopf algebra is Gorenstein and AS Gorenstein. However, there are Gorenstein Hopf algebras which are not necessarily AS Gorenstein.

\begin{lem}\label{lem:AS-Gor}
Let $H$ be a noetherian Hopf algebra with a bijective antipode. The following are equivalent:
\begin{enumerate}
\item $H$ is AS Gorenstein.
\item $H$ is Gorenstein, and both ${\rm Ext}_H^i(k, H)$ and ${\rm Ext}_{H^{op}}^i(k, H)$ are finite dimensional over $k$.
\end{enumerate}
\end{lem}

\begin{proof}
The implication (1)$\Rightarrow$(2) follows directly from the definition. For the converse, assume that ${\rm Gpd}_H(k)= d$ is finite. Then the self-injective dimension of $H$ equals $d$. Hence, (2)$\Rightarrow$(1) holds in view of \cite[Lemma 3.2]{BZ}.
\end{proof}

\begin{ex}\label{not-AS-G-eg}
Let $H$ be the free algebra $k\langle x_1, x_2, \dots, x_n \rangle$ with $n \geq 2$.
Then $H$ is a Hopf algebra with comultiplication, counit and antipode defined by
\[ \Delta(x_i) = 1 \otimes x_i + x_i \otimes 1, \qquad \varepsilon(x_i) = 0, \qquad S(x_i) = -x_i. \]
Noting that the global dimension of $H$ equals 1, $H$ is definitely a 1-Gorenstein algebra. However, $\Ext^{1}_{H}({}_{\varepsilon}k, H)$ is infinite dimensional. Hence, $H$ is not AS Gorenstein.
\end{ex}

It follows from \cite[Theorem 4.10]{Zh} that: Let $H$ and $L$ be monoidally Morita-Takeuchi equivalent Hopf algebras of type ${\rm FP}_\infty$ with bijective antipodes. If $H$ is AS regular of dimension $d$, then $L$ is AS Gorenstein of dimension $d$. Based on this, a question is raised in \cite{Zh}: for monoidally Morita-Takeuchi equivalent Hopf algebras $H$ and $L$, do they have the same self-injective dimension? In the following, an affirmative answer is given in (1) under the noetherian hypothesis, and (2) is of particular interest in view of the Question raised by Brown and Goodearl.

\begin{prop}\label{ID-H-L-cor}
Let $H$ and $L$ be monoidally Morita-Takeuchi equivalent noetherian Hopf algebras with bijective antipodes.
Then the following hold.
\begin{enumerate}
    \item ${\rm id}_H(H) = {\rm id}_L(L)$.
    \item $H$ is AS Gorenstein if and only if $L$ is AS Gorenstein.
\end{enumerate}
\end{prop}

\begin{proof}
The assertion (1) is an immediate consequence of Theorem \ref{mMTe-thm} and 
Corollary \ref{silp=spli-noe-Hopf-cor}.

(2) Let $B$ be an $L$-$H$-biGalois extension of $k$.
Without loss of generality, we may assume that ${\rm id}_H(H) = {\rm id}_L(L) = d < \infty$.
By the proof of \cite[Theorem 4.9. (2) and (4)]{Zh},
\[ \Ext^{i}_{B^e}(B, B^e) \cong \begin{cases}
    0, & i \neq d, \\
    B_{\mu}, & i = d,
\end{cases}\]
for some automorphism $\mu$ of $B$ if and only if 
\[ \Ext^{i}_{H^{op}}(k, H) \cong \begin{cases}
    0, & i \neq d, \\
    k, & i = d.
\end{cases} \]
Since $B$ is a right $L^{\mathrm{cop}}$-Galois extension of $k$, it follows that
\[ \Ext^{i}_{H^{op}}(k, H) \cong \begin{cases}
    0, & i \neq d, \\
    k, & i = d,
\end{cases} \text{ if and only if } \Ext^{i}_{L^{op}}(k, L) \cong \begin{cases}
    0, & i \neq d, \\
    k, & i = d.
\end{cases} \]
This implies that $H$ is AS Gorenstein of dimension $d$ if and only if $L$ is AS Gorenstein of dimension $d$.
\end{proof}

\begin{ex}\label{ex:Liu-Hopf-alg}
The Liu Hopf algebra $H_{\rm Liu}:= L(n, \xi)$ \cite{Liu} is generated by $x, g, g^{-1}$ with the relations
\[ gg^{-1} = g^{-1}g = 1, \quad xg = \xi gx, \quad x^n = 1 - g^n, \]
where $\xi$ is an $n$-th primitive root of unity. The comultiplication, counit and antipode on $H$ are defined by
\[ \Delta(g) = g \otimes g, \; \Delta(x) = x \otimes 1 + g \otimes x, \; \varepsilon(g) = 1, \; \varepsilon(x) = 0, \; S(g) = g^{-1}, \; S(x) = -g^{-1}x. \]

Let $L$ be the Hopf algebra which is generated by $y, g, g^{-1}$ with the relations
\[ gg^{-1} = g^{-1}g = 1, \quad yg = \xi gy, \quad y^n = 0. \]
The comultiplication, counit and antipode on $L$ are defined by
\[ \Delta(g) = g \otimes g, \; \Delta(y) = y \otimes 1 + g \otimes y, \; \varepsilon(g) = 1, \; \varepsilon(y) = 0, \; S(g) = g^{-1}, \; S(y) = -g^{-1}y. \]

As shown in \cite[Example 4.11]{Zh}, $L$ and $H_{\rm Liu}$ are monoidally Morita-Takeuchi equivalent.
The Liu Hopf algebra $H_{\rm Liu}$ has global dimension 1, whereas $L$ has infinite global dimension. However, Theorem \ref{mMTe-thm} shows that
 ${\rm Ggldim}(L) = {\rm Ggldim}(H_{\rm Liu}) = {\rm gldim}(H_{\rm Liu}) = 1$.
\end{ex}

\vspace{0.1in}

\section{Monoidal model category and tensor triangulated category}

\noindent  Let $A$ be a Gorenstein algebra, i.e. an algebra such that ${\rm l.Ggldim}(A)$ is finite. Analogous to \cite[Theorem 8.6]{H2}, there is a Gorenstein projective model structure on the category of left $A$-modules $_A\mathcal{M}$, which is defined explicitly as follows: 
\begin{itemize}
    \item The cofibrations (resp.\ the trivial cofibrations) are monomorphisms with cokernel being Gorenstein projective (resp. projective) $A$-module.
    \item The fibrations (resp. the trivial fibrations) are the epimorphisms (resp. the epimorphisms with kernel being modules of finite projective dimension).
    \item The weak equivalences are the maps which factor as a trivial cofibration followed by a trivial fibration.
\end{itemize}
The associated homotopy category ${\rm Ho}({_A}\mathcal{M})$, obtained by formally inverting the weak equivalences, is a triangulated category.
\vspace{0.1in}

Let $H$ be a Hopf algebra over a field $k$. The category $_H\mathcal{M}$ of left $H$-modules is naturally a monoidal category, with tensor product over $k$ induced by the comultiplication of $H$ and with the trivial module $k$ as the unit object. Recall that a monoidal model category is a model category equipped with a compatible monoidal structure; see \cite[Chapter~4]{H1} for details.

\begin{thm}\label{thm:m-m-cat}
Let $H$ be a Hopf algebra over a field $k$ with a bijective antipode. If ${\rm Gpd}_H({_\varepsilon}k)$ is finite, then the category $_H\mathcal{M}$ of left $H$-modules is a monoidal model category with respect to the Gorenstein projective model structure, and its homotopy category ${\rm Ho}({_H}\mathcal{M})$ is a tensor triangulated category.
\end{thm}

\begin{proof}
We begin by verifying the pushout product axiom. Recall that for two given morphisms $f: M\rightarrow N$ and $g: M' \rightarrow N'$, the
pushout product of $f$ and $g$ is defined to be the map
\[f\Box g: (M\otimes N')\bigsqcup_{M\otimes M'} (N\otimes M')\longrightarrow N\otimes N'.\]
Invoking Lemmas \ref{lem:Gproj}, \ref{lem:fact}(2) and  \ref{lem:fact'}(2), for any $H$-modules $M$ and $N$, if either $M$ or $N$ is Gorenstein projective (resp. projective), then $M\otimes  N$ is also a Gorenstein projective (resp. projective) $H$-module. Then, by a standard argument (cf. \cite[Theorem 7.2]{H2}), if both $f$ and $g$ are cofibrations, so is $f\Box g$; moreover, if either $f$ or $g$ is a weak equivalence, then so is $f\Box g$. The pushout product axiom implies that the tensor product functor (of two variables) $-\otimes-: {_H}\mathcal{M}\times {_H}\mathcal{M}\rightarrow {_H}\mathcal{M}$ is a Quillen bifunctor in the sense of \cite[Definition 4.2.1]{H1}. Hence, for any $H$-modules $M$ and $N$, the derived tensor product in the homotopy category ${\rm Ho}({_H}\mathcal{M})$ is defined by $M\otimes^{\mathbb L} N = Q(M)\otimes Q(N)$, where $Q(M)$ and $Q(N)$ denote cofibrant replacements of $M$ and $N$, respectively. These cofibrant replacements are precisely special right Gorenstein projective approximations of $M$ and $N$. Moreover, the operation $\otimes^{\mathbb L}$ satisfies the associativity isomorphism in the homotopy category ${\rm Ho}({_H}\mathcal{M})$. 

Regarding the unit axiom, we consider a short exact sequence of $H$-modules 
\[0\longrightarrow N\longrightarrow G\stackrel{q}\longrightarrow k\longrightarrow 0,\] 
where $G$ is Gorenstein projective and $N$ has finite projective dimension.
Such a sequence exists by \cite[Theorem 2.10]{Hol} because ${\rm Gpd}_H({_\varepsilon}k)<\infty$. In particular, $q: G\rightarrow k$ is a special right Gorenstein projective approximation of $k$. Under the Gorenstein projective model structure, $q: G \rightarrow k$ is therefore a trivial fibration, and $G:=Q(k)$ serves as a cofibrant replacement of $k$. Let $M$ be an arbitrary $H$-module.
We now consider the induced short exact sequences of left $H$-modules 
 \[0\rightarrow M\otimes N\rightarrow M\otimes G\stackrel{1\otimes q}\rightarrow M\rightarrow 0, \text{ and }\,\, 0\rightarrow N\otimes M\rightarrow G\otimes M\stackrel{q\otimes 1}\rightarrow M\rightarrow 0.\]
 In view of Lemma \ref{lem:pd},  both $M\otimes  N$ and $N\otimes M$ have finite projective dimension since so does $N$.  Hence both $1\otimes q: M\otimes G\rightarrow M\otimes k = M$ and  $q\otimes 1: G\otimes M\rightarrow k\otimes M = M$ are trivial fibrations, and in particular weak equivalences. Now the composition
 \[M \otimes^{\mathbb L} Q(k) = Q(M) \otimes Q(k) = Q(M) \otimes G
     \stackrel{1 \otimes q}{\longrightarrow} Q(M) \otimes k = Q(M)
     \stackrel{q_M}{\longrightarrow} M \]
is an isomorphism in ${\rm Ho}({_H}\mathcal{M})$ because both
$1 \otimes q$ and $q_M$ are weak equivalences. Similarly, we obtain an isomorphism $Q(k)\otimes^{\mathbb L} M\rightarrow M$  in ${\rm Ho}({_H}\mathcal{M})$. Consequently, the homotopy category ${\rm Ho}({_H}\mathcal{M})$ is a tensor triangulated category with tensor unit $Q(k)$; see \cite[Theorems 4.3.2, 6.6.4, and \S 7.1]{H1}.
\end{proof}

For any algebra $A$, the category of Gorenstein projective $A$-modules is a Frobenius category whose projective-injective objects are precisely the projective $A$-modules. Consequently, the stable category $\underline{\rm GProj}(A)$, obtained by modulo the morphisms factoring through projective modules, is a triangulated category.
If $A$ is of finite Gorenstein global dimension, then $\underline{\rm GProj}(A)$ is triangle-equivalent to the singularity category; see e.g. \cite{Buc}.

For the completeness, we briefly recall the description of the stable category $\underline{{\rm GProj}}(A)$ from the perspective of homotopy theory. Considering the Gorenstein projective model structure, ${\rm GProj}(A)$ consists of objects that are both cofibrant and fibrant, and the homotopy relation $\sim$ on ${\rm GProj}(A)$ is an equivalence relation. Let $f,g: M \rightarrow N$ be morphisms between Gorenstein projective modules, $\pi: P\rightarrow N$ be an epimorphism from a projective module $P$. Using the standard diagonal map $N\rightarrow N\times N$, we obtain a factorization
\[N\stackrel{i}\rightarrow N\oplus P\oplus P\stackrel{p}\rightarrow N\oplus N\]
where $i(n)=(n,0,0)$ is a trivial cofibration, and $p(n,x,y)=(n+\pi(x), n+\pi(y))$ is a fibration; thus $N\oplus P\oplus P$ is a path object for $N$. Then, $f\sim g$ if and only if there exists a map $(r,s,t):M\rightarrow N\oplus P\oplus P$ such that $p\circ(r,s,t)=(f,g)$; equivalently, if and only if the difference $f-g$ factors through the projective module $P$. Hence $\underline{{\rm GProj}}(A) = {\rm GProj}(A) /\sim $.

\vspace{0.1in}
\noindent  Recall that a tensor triangulated category $(\mathcal{K}, \otimes, \mathds{1})$ consists of a triangulated category $\mathcal{K}$ together with a monoidal structure $\otimes: \mathcal{K}\times \mathcal{K}\rightarrow \mathcal{K}$ and a unit object $\mathds{1}\in \mathcal{K}$; see e.g. \cite{Bal}. For the stable category $\underline{{\rm GProj}}(H)$, Proposition \ref{prop:f.d.H} shows that the trivial module $k$ lies in $\underline{{\rm GProj}}(H)$ if and only if $H$ is a finite dimensional Hopf algebra. Nevertheless, we have the following.

\begin{cor}\label{cor:tt-cat}
If ${\rm Gpd}_H({_\varepsilon}k)$ is finite, then the stable category $\underline{{\rm GProj}}(H)$ is a tensor triangulated category with unit object $Q(k)$. Moreover, if $H$ is cocommutative, then $\underline{{\rm GProj}}(H)$ is symmetric as a tensor triangulated category.
\end{cor}

\begin{proof}
Invoking Theorem \ref{thm:m-m-cat}, the category ${_H}\mathcal{M}$ carries a monoidal model structure. By a fundamental theorem on model categories (see e.g. \cite[Theorem 1.2.10]{H1}), the associated homotopy category ${\rm Ho}({_H}\mathcal{M})$ is triangle-equivalent to ${\rm GProj}(H) /\sim =\underline{{\rm GProj}}(H)$.   

The pushout product axiom implies that, for cofibrant objects (i.e. Gorenstein projective $H$-modules), the smash product is an invariant of the weak equivalence type, 
and therefore descends to a well-defined product on the homotopy category ${\rm Ho}({_H}\mathcal{M})$. In particular, for any Gorenstein projective $H$-modules $M$ and $N$, we have $M\otimes^{\mathbb L} N =M\otimes N$ in ${\rm Ho}({_H}\mathcal{M})$. Consequently, the tensor product in the stable category $\underline{{\rm GProj}}(H)$ is well-defined. Moreover, the unit axiom implies that, for any Gorenstein projective $H$-module $M$, \[M\otimes Q(k) =  M =  Q(k)\otimes M \]
in the stable category $\underline{{\rm GProj}}(H)$. Hence,  $\underline{{\rm GProj}}(H)$ is a tensor triangulated category with tensor unit $Q(k)$.
\end{proof}

Let $H$ and $L$ be monoidally Morita-Takeuchi equivalent Hopf algebras with bijective antipodes, and let $B$ be an $L$-$H$-biGalois extension of $k$.
By Schauenburg's structure theorem (see Theorem \ref{ff-Hopf-Galois-ext-Hopf-bimod-cat-str-thm}), there is an equivalence 
\[ \xymatrix{
(-)^{co H}: {_B\!\M_B^H}\ar@<0.5ex>[r] &{_L\!\M}: - \otimes B \ar@<0.5ex>[l]   } \]
of categories. Theorem \ref{mMTe-thm} shows that the finiteness of the Gorenstein global dimensions of $H$, $L$ and $B^e$ are equivalent. Assuming this finiteness, the Gorenstein projective model structure exists on the category ${_L\!\M}$ of $L$-modules. Moreover, by Proposition \ref{Gpd-L-RB-prop} we obtain the following consequences. For details on Quillen equivalence, we refer to \cite[\S 1.3]{H1}.

\begin{lem}\label{lem:mcs-B-B} 
Assume that  ${\rm Gpd}_H({_\varepsilon}k)$ is finite.
\begin{enumerate}
    \item For any morphism $f: M\rightarrow N$ in ${_B\!\M_B^H}$, we define that it is a cofibration (resp. fibration, weak equivalence) provided that $f^{co H}: M^{co H}\rightarrow N^{co H}$ is a cofibration (resp. fibration, weak equivalence) in the Gorenstein projective model structure on $_L\!\M$. This construction establishes a Gorenstein projective model structure on ${_B\!\M_B^H}$.
    \item The adjoint pair $((-)^{co H}, - \otimes B)$ is a Quillen equivalence between the Gorenstein projective model categories ${_B\!\M_B^H}$ and $_L\!\M$.
\end{enumerate}
\end{lem}

 Note that ${_B\!\M^H_B}$ is a monoidal category equipped with the tensor product over $B$. We conclude this section by showing that ${_B\!\M_B^H}$ is a monoidal model category.

\begin{thm}\label{thm:mMTe-mmc}
Let $H$ and $L$ be monoidally Morita-Takeuchi equivalent Hopf algebras with bijective antipodes, and let $B$ be an $L$-$H$-biGalois extension of $k$. Assume that ${\rm Gpd}_H({_\varepsilon}k)$ is finite. Then
${_B\!\M_B^H}$ is a monoidal model category with respect to the Gorenstein projective model structure. Consequently, the stable category of Gorenstein projective objects in ${_B\!\M_B^H}$ is a tensor triangulated category, with the tensor product taken over $B$.
\end{thm}

\begin{proof}
Considering the model structure on ${_B\!\M_B^H}$ established in Lemma \ref{lem:mcs-B-B}, for any $M \in {_B\!\M^H_B}$, it is a cofibrant object (i.e. Gorenstein projective object) if and only if $M^{co H}$ is a cofibrant object in  $_L\!\M$ (i.e. a Gorenstein projective $L$-module); it is a trivial object if and only if $M^{co H}$ has finite projective dimension. For any $M, N \in {_B\!\M^H_B}$, we have
\[M \otimes_B N = (M^{co H} \otimes B)\otimes_B (N^{co H}\otimes B) = (M^{co H} \otimes N^{co H}) \otimes B. \]
If $M$ or $N$ is a Gorenstein projective object in ${_B\!\M^H_B}$, then $M^{co H}$ or $N^{co H}$ is a Gorenstein projective $L$-module. By Lemma \ref{lem:Gproj}, the tensor product $M^{co H} \otimes N^{co H}$ is again Gorenstein projective over $L$; 
Proposition \ref{Gpd-L-RB-prop} therefore implies that $M \otimes_B N \cong (M^{co H} \otimes N^{co H}) \otimes B$ is Gorenstein projective in ${_B\!\M^H_B}$. A parallel argument shows that when  $M$ or $N$ is projective in ${_B\!\M^H_B}$, the tensor product $M \otimes_B N \cong (M^{co H} \otimes N^{co H}) \otimes B$ is also projective in ${_B\!\M^H_B}$. In view of \cite[Theorem 7.2]{H2}, these facts together verify the pushout product axiom for the Gorenstein projective model structure on ${_B\!\M^H_B}$.
 
In general, the tensor unit $B$ is not necessarily a Gorenstein projective object in   ${_B\!\M^H_B}$. To verify the unit axiom, consider the short exact sequence of $L$-modules 
\[0\longrightarrow K\longrightarrow Q(k)\stackrel{q}\longrightarrow k\longrightarrow 0,\] 
where $Q(k)$ is a cofibrant replacement (i.e. a special right Gorenstein projective approximation) of $k$ and $K$ has finite projective dimension.
The existence of such a sequence follows from \cite[Theorem 2.10]{Hol} because ${\rm Gpd}_L({_\varepsilon}k)={\rm Gpd}_H({_\varepsilon}k)$ is finite. Then, by applying $-\otimes B$, we get a short exact sequence 
\[0\longrightarrow K\otimes B\longrightarrow Q(k)\otimes B\stackrel{q'}\longrightarrow B\longrightarrow 0\]
in ${_B\!\M_B^H}$, where $Q(k)\otimes B$ is Gorenstein projective by Proposition \ref{Gpd-L-RB-prop}. Now for any  $M\in {_B\!\M_B^H}$, we have 
  \[M\otimes_B (K\otimes B)= (M^{co H}\otimes K)\otimes B, \quad (K\otimes B)\otimes_B M=  (K\otimes M^{co H})\otimes B.\]
In view of Lemma \ref{lem:pd}, both $M^{co H}\otimes K$ and $K\otimes M^{co H}$ have  finite projective dimension as $L$-modules since $K$ does so. Then, $M\otimes_B (K\otimes B)$ and $(K\otimes B)\otimes_B M$ also have finite projective dimension. Therefore, ${\rm id}_M\otimes q'$ and $q'\otimes{\rm id}_M$ are weak equivalences, which establishes the unit axiom. 

Consequently, the homotopy category ${\rm Ho}({_B\!\M_B^H})$ is a tensor triangulated category. Analogous to Corollary \ref{cor:tt-cat}, we can prove that the stable category of Gorenstein projective objects in ${_B\!\M_B^H}$ is a tensor triangulated category, with the tensor product taken over $B$. 
\end{proof}

To illustrate the preceding results, we return to the pair of monoidally Morita--Takeuchi equivalent Hopf algebras considered in Example~\ref{ex:Liu-Hopf-alg}. 

\begin{ex}
Let $H_{\mathrm{Liu}}=L(n,\xi)$ and $L$ be the monoidally
Morita--Takeuchi equivalent Hopf algebras from
Example~\ref{ex:Liu-Hopf-alg}. Recall that ${\rm gldim}(H_{\mathrm{Liu}})=1$, ${\rm gldim}(L)=\infty$, 
whereas ${\rm Ggldim}(H_{\mathrm{Liu}})
= {\rm Ggldim}(L) = 1$.
By Theorem~\ref{thm:m-m-cat}, both
${}_{H_{\mathrm{Liu}}}\mathcal{M}$ and ${}_L\mathcal{M}$ admit Gorenstein projective monoidal model structures, and their homotopy categories are tensor triangulated.

For $H_{\mathrm{Liu}}$, the finiteness of the global dimension implies that every Gorenstein projective module is projective. Hence, one has $\underline{\rm GProj}(H_{\mathrm{Liu}})=0$.
In contrast, $L$ has infinite global dimension but Gorenstein global dimension one. In particular, $L$ admits non-projective Gorenstein projective modules. Thus the tensor triangulated stable category $\underline{\rm GProj}(L)$ is nonzero.

Moreover, let $B$ be an $L$-$H_{\mathrm{Liu}}$-biGalois object implementing
the monoidal Morita--Takeuchi equivalence. Then
Theorem~\ref{thm:mMTe-mmc} shows that ${}_B\!\M_B^{H_{\mathrm{Liu}}}$
also carries a Gorenstein projective monoidal model structure, and the stable
category of its Gorenstein projective objects is tensor triangulated with
respect to the tensor product over $B$.
\end{ex}

\vspace{0.1in}

\section*{Acknowledgements}
\noindent The authors are grateful to the anonymous referee for the careful reading of the manuscript and for the constructive comments and suggestions, which helped improve both the clarity and the presentation of the paper. W. Ren was supported by the Natural Science Foundation of
Chongqing, China (No.\ CSTB2025NSCQ-GPX1014).
R. Zhu was supported by the National Natural Science Foundation of China (No. 12301052).

\vspace{0.1in}

{\footnotesize \noindent Wei Ren\\
 School of Mathematical Sciences, Chongqing Normal University, Chongqing 401331, PR China\\
 E-mail: {\tt wren$\symbol{64}$cqnu.edu.cn}}

\vspace{0.1in}

{\footnotesize \noindent Ruipeng Zhu\\
School of Mathematics, Shanghai University of Finance and Economics, Shanghai 200433, PR China \\
E-mail: {\tt zhuruipeng$\symbol{64}$sufe.edu.cn}}

\vspace{0.1in}

\begin{thebibliography}{99}

\bibitem{AB} M. Auslander and M. Bridger, Stable Module Theory, American Mathematical Society, Providence, RI, Memoirs of the American Mathematical Society no. 94, 1969.
	
\bibitem{Bal} P. Balmer, Tensor triangular geometry, In International Congress of
	Mathematicians, Hyderabad (2010), Vol. II, pages 85--112. Hindustan Book Agency, 2010.

 \bibitem{BM} D. Bennis and N. Mahdou, Global Gorenstein dimensions, 
Proc. Amer. Math. Soc. {\bf 138} (2010), 461--465.
	
\bibitem{Bi} J. Bichon, On the monoidal invariance of the cohomological
	dimension of Hopf algebras, C. R. Math. Acad. Sci. Paris
	{\bf 360} (2022), 561--582.

\bibitem{Br} K.A. Brown, Representation theory of Noetherian Hopf algebras satisfying a polynomial identity, in: Trends in the Representation Theory of Finite-Dimensional Algebras, Seattle, WA, 1997, in: Contemp. Math., vol. 229, Amer. Math. Soc., Providence, RI, 1998, pp. 49--79.

\bibitem{BG} K.A. Brown, K.R. Goodearl, Homological aspects of Noetherian PI Hopf algebras of irreducible modules and maximal dimension, J. Algebra 198 (1997), 240--265.
	
\bibitem{BZ} K.A. Brown and J.J. Zhang, Dualising complexes and twisted Hochschild
	(co)homology for Noetherian Hopf algebras, J. Algebra {\bf 320}(5) (2008), 1814--1850.

 \bibitem{BW} T. Brzezinski and R. Wisbauer, Corings and Comodules, London Mathematical Society Lecture Note Series, vol. 309, Cambridge University Press, Cambridge, 2003.

 \bibitem{Buc} 
R.-O. Buchweitz, Maximal Cohen--Macaulay Modules and Tate Cohomology, Mathematical Surveys and Monographs, vol. 262,
American Mathematical Society, Providence, RI, 2021.
	
\bibitem{CR}
X.-W. Chen and W. Ren, Frobenius functors and Gorenstein homological properties, J. Algebra {\bf 610} (2022), 18--37.

\bibitem{CET} L.W. Christensen, S. Estrada and P. Thompson,
Five theorems on Gorenstein global dimensions,
Contemp. Math., {\bf 785} American Mathematical Society, Providence, RI, 2023, 67--78.

\bibitem{E1} I. Emmanouil, On certain cohomological invariants of groups, Adv. Math. {\bf 225} (2010), no. 6, 3446--3462.

 \bibitem{E2} I. Emmanouil, On the finiteness of Gorenstein homological dimensions, J. Algebra {\bf 372} (2012), 376--396.

\bibitem{ET} I. Emmanouil and O. Talelli, Finiteness criteria in Gorenstein homological algebra, Trans. Amer. Math. Soc. {\bf 366} (2014), no. 12, 6329--6351.

\bibitem{ET1} I. Emmanouil and O. Talelli, Gorenstein dimension and group cohomology with group ring coefficients, J.\ London Math.\ Soc.\ {\bf 97} (2018), 306--324.
 
\bibitem{ET3} I. Emmanouil and O. Talelli, On the Gorenstein cohomological dimension of group extensions, J. Algebra {\bf 605} (2022), 403--428.
   
\bibitem{EJ} E.E. Enochs and O.M.G. Jenda, Relative Homological Algebra, De Gruyter Expositions in Mathematics no. 30, New York: Walter De Gruyter, 2000.

 \bibitem{GG} T.V. Gedrich and K.W. Gruenberg, Complete cohomological functors on groups, Topology Appl. {\bf 25} (1987), 203--223.
	
\bibitem{Hol} H. Holm, Gorenstein homological dimensions,
	J. Pure Appl. Algebra {\bf 189} (2004), 167--193.

\bibitem{H1}   M. Hovey, Model Categories, Mathematical Surveys and Monographs vol. 63, American Mathematical Society, 1999.

\bibitem{H2}  M. Hovey, Cotorsion pairs, model category structures, and representation theory, Math.\ Z.\ {\bf 241} (2002), 553--592.

\bibitem{LeC} J. Le Clainche, Homological duality and exact sequences of Hopf algebras, arXiv: 2501.14800.

\bibitem{LZ} C.-H. Liu and J.J. Zhang, Artinian Hopf algebras are finite dimensional, Proc. Amer. Math. Soc. {\bf 135}(6) (2007), 1679--1680.

\bibitem{Liu} G. Liu, On Noetherian affine prime regular Hopf algebras of Gelfand-Kirillov dimension 1, Proc. Amer. Math. Soc. {\bf 137}(3) (2009), 777--785.

\bibitem{LL} M.E. Lorenz and M. Lorenz, On crossed products
of Hopf algebras, Proc. Amer. Math. Soc. {\bf 123}(1)
(1995), 33--38.

\bibitem{MW} A. Masuoka, D. Wigner, Faithful flatness of Hopf algebras, J. Algebra {\bf 170} (1994), no. 1, 156--164.
 
\bibitem{M} S. Montgomery, Hopf Algebras and Their Actions on Rings, in: CBMS Reg. Conf. Ser. Math., vol. 82, American Mathematical Society, Providence, RI, 1993.

\bibitem{MS} E. F. Müller, H.-J. Schneider, Quantum homogeneous spaces with faithfully flat module structures, Israel J. Math. {\bf 111} (1999), 157--190.

\bibitem{Q} D. Quillen, Homotopical Algebra, Lecture Notes
in Mathematics no.\ {\bf 43}, Springer-Verlag, 1967.

\bibitem{R1} D. E. Radford, Operators on Hopf algebras, Amer. J. Math. {\bf 99} (1977), no. 1, 139--158.

\bibitem{R2} D. E. Radford, Pointed Hopf algebras are free over Hopf subalgebras, J. Algebra {\bf 45} (1977), no. 2, 266--273.

\bibitem{S1} P. Schauenburg, Hopf bi-Galois extensions, Comm. Algebra, {\bf 24} (1996), 3797--3825.
 
\bibitem{S2} P. Schauenburg, Bialgebras over noncommutative rings and a structure theorem for Hopf bimodules, Appl. Categ. Structures, {\bf 6} (1998), 193--222.
 
\bibitem{S} H. Schneider, Principal homogeneous spaces for arbitrary Hopf algebras, Israel J. Math. {\bf 72} (1990), no. 1-2, 167--195.

\bibitem{SS} P. Schauenburg and H. Schneider, On generalized Hopf Galois extensions, J. Pure Appl. Algebra {\bf 202} (2005), no. 1-3, 168--194.

\bibitem{Sk1} S. Skryabin, New results on the bijectivity of antipode of a Hopf algebra, J. Algebra {\bf 306}(2) (2006), 622--633.

\bibitem{Ste} D. Stefan, Hochschild cohomology on Hopf Galois extensions, J. Pure Appl. Algebra {\bf 103} (1995), no. 2, 221--233.
	
\bibitem{WYZ} X. Wang, X. Yu and Y. Zhang, Calabi-Yau property under monoidal Morita-Takeuchi equivalence, Pacific J. Math. {\bf 290}(2) (2017), 481--510.

\bibitem{W} C.A. Weibel, An Introduction to Homological Algebra, Cambridge Stud. Adv. Math., vol. 38, Cambridge Univ. Press, Cambridge, 1994.

\bibitem{WZ} Q.-S. Wu and J.J. Zhang, Noetherian PI Hopf algebras are Gorenstein, Trans. Amer. Math. Soc. {\bf 355}(3) (2002), 1043--1066.

\bibitem{Zh} R. Zhu, Artin-Schelter Gorenstein property of Hopf Galois extensions, J. Pure Appl. Algebra, 229 (2025), Paper No. 108123.

\end{thebibliography}
\end{document}